\documentclass{article}

\usepackage{amssymb}
\usepackage{amsthm}
\usepackage{capt-of}
\usepackage{commath}
\usepackage[margin=1in]{geometry}
\usepackage[hidelinks]{hyperref}
\usepackage{mathtools}

\newtheorem{n}{}[section]

\theoremstyle{plain}

\newtheorem{conjecture}[n]{Conjecture}
\newtheorem{corollary}[n]{Corollary}
\newtheorem{lemma}[n]{Lemma}
\newtheorem{proposition}[n]{Proposition}
\newtheorem{theorem}[n]{Theorem}

\theoremstyle{definition}
\newtheorem{example}[n]{Example}
\newtheorem{remark}[n]{Remark}

\DeclareFontFamily{U}{wncyr}{}
\DeclareFontShape{U}{wncyr}{m}{n}{<->wncyr10}{}
\DeclareSymbolFont{cyr}{U}{wncyr}{m}{n}
\DeclareMathSymbol{\Sha}{\mathord}{cyr}{"58}

\renewcommand{\a}{\mathrm{a}}
\newcommand{\AAA}{\mathcal{A}}
\newcommand{\BSD}{\mathrm{BSD}}
\renewcommand{\c}{\mathrm{c}}
\newcommand{\CC}{\mathbb{C}}
\renewcommand{\d}{\mathrm{d}}
\newcommand{\FF}{\mathbb{F}}
\newcommand{\Fr}{\mathrm{Fr}}
\newcommand{\G}{\mathrm{G}}
\newcommand{\Gal}{\mathrm{Gal}}
\newcommand{\GL}{\mathrm{GL}}
\newcommand{\h}{\mathrm{h}}
\renewcommand{\H}{\mathrm{H}}
\newcommand{\HHH}{\mathcal{H}}
\newcommand{\im}{\mathrm{im}}

\renewcommand{\L}{\mathrm{L}}
\newcommand{\LLL}{\mathcal{L}}
\newcommand{\M}{\mathrm{M}}
\newcommand{\Nm}{\mathrm{Nm}}
\newcommand{\NN}{\mathbb{N}}
\newcommand{\ord}{\mathrm{ord}}
\newcommand{\PGL}{\mathrm{PGL}}
\newcommand{\PSL}{\mathrm{PSL}}
\newcommand{\QQ}{\mathbb{Q}}
\newcommand{\Reg}{\mathrm{Reg}}
\newcommand{\rk}{\mathrm{rk}}
\newcommand{\RR}{\mathbb{R}}
\renewcommand{\S}{\mathrm{S}}
\newcommand{\SSS}{\mathcal{S}}
\newcommand{\SL}{\mathrm{SL}}
\newcommand{\Tam}{\mathrm{Tam}}
\newcommand{\tor}{\mathrm{tor}}
\newcommand{\tr}{\mathrm{tr}}

\newcommand{\X}{\mathrm{X}}
\newcommand{\ZZ}{\mathbb{Z}}

\newcommand{\abr}[2][-1]{\!\ifthenelse{\equal{#1}{-1}}{\left\langle#2\right\rangle}{}\ifthenelse{\equal{#1}{0}}{\langle#2\rangle}{}\ifthenelse{\equal{#1}{1}}{\bigl\langle#2\bigr\rangle}{}}
\newcommand{\br}{\!\del}
\let\cb\cbr\renewcommand{\cbr}{\!\cb}
\newcommand{\fbr}[2][-1]{\!\ifthenelse{\equal{#1}{-1}}{\left\lfloor#2\right\rfloor}{}\ifthenelse{\equal{#1}{0}}{\lfloor#2\rfloor}{}\ifthenelse{\equal{#1}{1}}{\bigl\lfloor#2\bigr\rfloor}{}}
\let\sb\sbr\renewcommand{\sbr}{\!\sb}
\newcommand{\st}{\ \middle| \ }
\newcommand{\twobytwo}[4]{\begin{pmatrix} #1 & #2 \\ #3 & #4 \end{pmatrix}}
\newcommand{\twobytwosmall}[4]{\begin{psmallmatrix} #1 & #2 \\ #3 & #4 \end{psmallmatrix}}

\title{L-values of elliptic curves twisted by cubic characters}
\author{David Kurniadi Angdinata}

\begin{document}

\maketitle

\begin{abstract}
Given a rational elliptic curve $ E $ of analytic rank zero, its L-function can be twisted by an even primitive Dirichlet character $ \chi $ of order $ q $, and in many cases its associated central algebraic L-value $ \LLL\br{E, \chi} $ is known to be integral. This paper derives some arithmetic consequences from a congruence between $ \LLL\br{E, 1} $ and $ \LLL\br{E, \chi} $ arising from this integrality, with an emphasis on cubic characters $ \chi $. These include $ q $-adic valuations of the denominator of $ \LLL\br{E, 1} $, determination of $ \LLL\br{E, \chi} $ in terms of Birch--Swinnerton-Dyer invariants, and asymptotic densities of $ \LLL\br{E, \chi} $ modulo $ q $ by varying $ \chi $.
\end{abstract}

\tableofcontents

\subsection*{Acknowledgements}

I would like to thank Vladimir Dokchitser for suggesting the problem and guidance throughout, and Harry Spencer for feedback on a prior draft. I would also like to thank Peiran Wu and Jacob Fjeld Grevstad for assistance on some of the group theory arguments. This work was supported by the Engineering and Physical Sciences Research Council [EP/S021590/1], the EPSRC Centre for Doctoral Training in Geometry and Number Theory (The London School of Geometry and Number Theory), University College London.

\section{Introduction}

The Hasse--Weil L-function $ \L\br{E, s} $ of an elliptic curve $ E $ over $ \QQ $ can be twisted by a primitive Dirichlet character $ \chi $ to get a twisted L-function $ \L\br{E, \chi, s} $. The algebraic and analytic properties of these L-functions are studied extensively in the literature, and they are the subject of many problems in the arithmetic of elliptic curves. Classically, the Birch--Swinnerton-Dyer conjecture relates $ \L\br{E, s} $ to certain algebraic invariants of $ E $, which encode important global arithmetic information of $ E $. When $ E $ is base changed to a cyclotomic extension $ K $ of $ \QQ $, Artin's formalism for L-functions says that $ \L\br[0]{E / K, s} $ decomposes into a product of twisted L-functions $ \L\br{E, \chi, s} $ over all Dirichlet characters $ \chi $ that factor through $ K $, so the behaviour of $ \L\br[0]{E / K, s} $, and hence the arithmetic of $ E / K $, is completely governed by these twisted L-functions.

\pagebreak

By the modularity theorem, $ \L\br{E, \chi, s} $ admits an analytic continuation to all $ s \in \CC $, so its value at $ s = 1 $ is well-defined. When it does not vanish, $ \L\br{E, \chi, 1} $ can be normalised by periods to get an algebraic twisted L-value $ \LLL\br{E, \chi} $. Under some mild hypotheses on $ E $ and $ \chi $, Wiersema--Wuthrich proved that $ \LLL\br{E, \chi} \in \ZZ\sbr{\zeta_q} $ for non-trivial primitive Dirichlet characters $ \chi $ of order $ q $ \cite[Theorem 2]{WW22}, by expressing $ \LLL\br{E, \chi} $ in terms of Manin's modular symbols. Parts of their argument can be adapted to obtain a congruence between $ \LLL\br{E, \chi} $ and $ \LLL\br{E, 1} $ modulo the prime $ \abr{1 - \zeta_q} $ in $ \ZZ\sbr{\zeta_q} $ above $ q $, and this turns out to have arithmetic consequences for $ \LLL\br{E, \chi} $. This paper derives a few consequences for when $ q = 3 $.

As one consequence, this addresses a problem of Dokchitser--Evans--Wiersema, in an attempt to obtain a Birch--Swinnerton-Dyer-type formula for twisted L-values \cite{DEW21}. Under standard arithmetic conjectures, they expressed the norm of $ \LLL\br{E, \chi} $ in terms of the Birch--Swinnerton-Dyer quotients $ \BSD\br{E} $ and $ \BSD\br[0]{E / K} $, where $ K $ is the number field that $ \chi $ factors through, and deduced furthermore that there are only finitely many possibilities for the ideal generated by $ \LLL\br{E, \chi} $ \cite[Theorem 38]{DEW21}. A subsequent refinement by Burns--Castillo recovers the ideal factorisation of $ \LLL\br{E, \chi} $ in terms of the Galois module structure of the Tate--Shafarevich group \cite[Remark 7.4]{BC21}, but the precise value of $ \LLL\br{E, \chi} $ is only predicted up to a unit in $ \ZZ\sbr{\zeta_q} $. In particular, there are pairs of elliptic curves with identical Birch--Swinnerton-Dyer quotients but with distinct twisted L-values \cite[Section 4]{DEW21}. The following result determines this unit for cubic characters, and hence explains their examples, under an assumption on the Manin constant $ \c_0\br{E} $.

\begin{theorem}[Corollary \ref{cor:cubic}]
Let $ E $ be an elliptic curve over $ \QQ $ of conductor $ N $, and let $ \chi $ be a cubic Dirichlet character of odd prime conductor $ p \nmid N $ such that $ 3 \nmid \c_0\br{E}\BSD\br{E}\#E\br{\FF_p} $. Assume that Stevens's conjecture holds for $ E $, and that the Birch--Swinnerton-Dyer conjecture holds for $ E $ over number fields. Then
$$ \LLL\br{E, \chi} = u \cdot \overline{\chi\br{N}}\dfrac{\sqrt{\BSD\br[0]{E / K}}}{\sqrt{\BSD\br{E}}}, $$
for some sign $ u = \pm1 $, chosen such that
$$ u \equiv -\#E\br{\FF_p}\dfrac{\sqrt{\BSD\br{E}}^3}{\sqrt{\BSD\br[0]{E / K}}} \mod 3. $$
\end{theorem}

In particular, all of the arithmetically identical pairs of elliptic curves considered by Dokchitser--Evans--Wiersema have distinct twisted L-values simply because their number of points modulo $ p $ are distinct modulo $ 3 $. Section \ref{sec:unit} proves this result and clarifies their examples in more detail.

As another consequence, this explains some of the patterns in Kisilevsky--Nam, where they made heuristic predictions on the asymptotic distribution of twisted L-values \cite{KN22}. To support these predictions, they considered six elliptic curves $ E $ of analytic rank zero and five positive integers $ q $, and numerically computed the norms of L-values of $ E $ twisted by primitive Dirichlet characters $ \chi $ of conductor $ p $ and order $ q $, for each of the thirty pairs $ \br{E, q} $ and over millions of positive integers $ p $ \cite[Section 7]{KN22}. For each pair $ \br{E, q} $, they added a normalisation factor to $ \LLL\br{E, \chi} $ to obtain a real value $ \LLL^+\br{E, \chi} $, and empirically determined the greatest common divisor $ \gcd_{E, q} $ of the norms of $ \LLL^+\br{E, \chi} $ by varying over all $ p $. Upon dividing these norms by $ \gcd_{E, q} $ to get an integer $ \widetilde{\LLL}^+\br{E, \chi} $, they observed that these integers have unexpected biases when reduced modulo $ q $. The following result explains these biases modulo $ 3 $ for many elliptic curves.

\begin{theorem}[Proposition \ref{prop:cubic}]
Let $ E $ be an elliptic curve over $ \QQ $ of conductor $ N $ and discriminant $ \Delta = \pm N^n $ for some $ 3 \nmid n $ with no rational $ 3 $-isogeny, such that $ 3 \nmid \c_0\br{E} $ and $ 3 \nmid \gcd_{E, 3} $, and let $ \chi $ be a cubic Dirichlet character of odd prime conductor $ p \nmid N $. Then
$$ \widetilde{\LLL}^+\br{E, \chi} \equiv
\begin{cases}
0 \mod 3 & \text{if} \ \#E\br{\FF_p} \equiv 0 \mod 3 \\
2 \mod 3 & \text{if} \ \#E\br{\FF_p} \equiv 1 \mod 3 \ \text{and} \ p \ \text{splits completely in the $ 3 $-division field of} \ E \\
1 \mod 3 & \text{otherwise}
\end{cases}.
$$
\end{theorem}

In particular, the elliptic curves given by the Cremona labels 11a1, 15a1, and 17a1 considered by Kisilevsky--Nam satisfy these assumptions, and their normalised twisted L-values modulo $ 3 $ are as predicted. Section \ref{sec:kn22} defines their normalisation for $ \LLL^+\br{E, \chi} $ and proves this result under slightly relaxed assumptions.

\pagebreak

Now these biases can be quantified asymptotically by considering their natural densities when reduced modulo $ q $. Observe that for any Dirichlet character $ \chi $ of prime order $ q $ of the same conductor $ p $, the twisted L-values $ \LLL\br{E, \chi} $ are congruent modulo $ \abr{1 - \zeta_q} $ simply because $ \chi \equiv 1 \mod \abr{1 - \zeta_q} $, so the natural density of $ \LLL\br{E, \chi} $ modulo $ \abr{1 - \zeta_q} $ only depends on $ p $ and not $ \chi $. More precisely, let $ X_{E, q}^{< n} $ be the set of equivalence classes of Dirichlet characters of order $ q $ and conductor less than $ n $ but coprime to that of $ E $, where two Dirichlet characters in $ X_{E, q}^{< n} $ are equivalent if they have the same conductor. Then define the residual densities $ \delta_{E, q} $ of $ \LLL\br{E, \chi} $ to be the natural densities of $ \LLL\br{E, \chi} $ modulo $ \abr{1 - \zeta_q} $, or in other words
$$ \delta_{E, q}\br{\lambda} \coloneqq \lim_{n \to \infty} \dfrac{\#\cbr{\chi \in X_{E, q}^{< n} \st \LLL\br{E, \chi} \equiv \lambda \mod \abr{1 - \zeta_q}}}{\#X_{E, q}^{< n}}, \qquad \lambda \in \FF_q, $$
if such a limit exists. The following result classifies the possible ordered triples $ \br{\delta_{E, 3}\br{0}, \delta_{E, 3}\br{1}, \delta_{E, 3}\br{2}} $ of residual densities in terms of $ \BSD\br{E} $, the torsion subgroup $ \tor\br{E} $, and the mod-$ 9 $ Galois image $ \im\overline{\rho_{E, 9}} $.

\begin{theorem}[Theorem \ref{thm:density}]
Let $ E $ be an elliptic curve over $ \QQ $ such that $ 3 \nmid \c_0\br{E} $ and $ \L\br{E, 1} \ne 0 $. Assume that the Birch--Swinnerton-Dyer conjecture holds for $ E $. Then the ordered triple $ \br{\delta_{E, 3}\br{0}, \delta_{E, 3}\br{1}, \delta_{E, 3}\br{2}} $ only depends on $ \BSD\br{E} $ and on $ \im\overline{\rho_{E, 9}} $, and can only be one of
$$ \br{1, 0, 0}, \ \br{\tfrac{3}{8}, \tfrac{3}{8}, \tfrac{1}{4}}, \ \br{\tfrac{3}{8}, \tfrac{1}{4}, \tfrac{3}{8}}, \ \br{\tfrac{1}{2}, \tfrac{1}{2}, 0}, \ \br{\tfrac{1}{2}, 0, \tfrac{1}{2}}, \ \br{\tfrac{1}{8}, \tfrac{3}{4}, \tfrac{1}{8}}, $$
$$ \br{\tfrac{1}{8}, \tfrac{1}{8}, \tfrac{3}{4}}, \ \br{\tfrac{1}{4}, \tfrac{1}{2}, \tfrac{1}{4}}, \ \br{\tfrac{1}{4}, \tfrac{1}{4}, \tfrac{1}{2}}, \ \br{\tfrac{5}{9}, \tfrac{2}{9}, \tfrac{2}{9}}, \ \br{\tfrac{1}{3}, \tfrac{2}{3}, 0}, \ \br{\tfrac{1}{3}, 0, \tfrac{2}{3}}. $$
In particular, the ordered triple $ \br{\delta_{E, 3}\br{0}, \delta_{E, 3}\br{1}, \delta_{E, 3}\br{2}} $ can be precisely determined as follows.
\begin{itemize}
\item If $ \ord_3\br{\BSD\br{E}} = 0 $ and $ 3 \nmid \#\tor\br{E} $, then $ \br{\delta_{E, 3}\br{0}, \delta_{E, 3}\br{1}, \delta_{E, 3}\br{2}} $ is given by Table \ref{tab:mod3}.
\item If $ \ord_3\br{\BSD\br{E}} = -1 $, then $ \br{\delta_{E, 3}\br{0}, \delta_{E, 3}\br{1}, \delta_{E, 3}\br{2}} $ is given by Table \ref{tab:3adic}.
\item Otherwise, $ \delta_{E, 3}\br{0} = 1 $ and $ \delta_{E, 3}\br{1} = \delta_{E, 3}\br{2} = 0 $.
\end{itemize}
\end{theorem}

Note that the aforementioned normalisation factors for twisted L-values are not present here, so the resulting ordered triples of residual densities will be different from that of Kisilevsky--Nam. Section \ref{sec:density} proves this result and outlines the general procedure for higher order characters.

This classification builds upon the fact that $ \ord_3\br{\BSD\br{E}} \ge -1 $, which might be interesting as a standalone result. In a seminal paper quantifying the cancellations between $ \tor\br{E} $ and the Tamagawa product $ \Tam\br{E} $, Lorenzini proved that if $ q \mid \#\tor\br{E} $ for some prime $ q > 3 $, then $ q \mid \Tam\br{E} $ with finitely many explicit exceptions \cite[Proposition 1.1]{Lor11}. In particular, when $ E $ has analytic rank zero, the denominator $ \#\tor\br{E}^2 $ of the rational number $ \BSD\br{E} $ shares factors with $ \Tam\br{E} $ in its numerator, so $ \ord_q\br{\BSD\br{E}} \ge -1 $ for any prime $ q > 3 $. On the other hand, he noted that there is an explicit family of elliptic curves with $ \#\tor\br{E} = 3 $ but with no such cancellations \cite[Lemma 2.26]{Lor11}, another family of which was also given by Barrios--Roy \cite[Corollary 5.1]{BR22}. Subsequently, Melistas showed that these cancellations may instead occur between $ \tor\br{E} $ and the Tate--Shafarevich group $ \Sha\br{E} $ in the numerator of $ \BSD\br{E} $, and hence $ \ord_3\br{\BSD\br{E}} \ge -1 $, except possibly for elliptic curves with certain reduction types \cite[Theorem 1.4]{Mel23}. He then observed that there are again explicit exceptions, and in all these exceptions $ \c_0\br{E} = 3 $ \cite[Example 3.8]{Mel23}, but did not explain this coincidence. The following result gives a lower bound for the odd part of the denominator of Birch--Swinnerton-Dyer quotients for elliptic curves with analytic rank zero.

\begin{theorem}[Theorem \ref{thm:valuation}]
Let $ E $ be an elliptic curve over $ \QQ $ such that $ \L\br{E, 1} \ne 0 $, and let $ q \nmid \c_0\br{E} $ be an odd prime. Assume that the Birch--Swinnerton-Dyer conjecture holds for $ E $. If $ q \mid \#\tor\br{E} $, then $ q \mid \Tam\br{E}\#\Sha\br{E} $. In particular, $ \ord_q\br{\BSD\br{E}} \ge -1 $.
\end{theorem}

Section \ref{sec:valuation} states this result in terms of $ \LLL\br{E, 1} $ and proves it in slightly larger generality. Note that there will be a stronger statement for elliptic curves with no rational $ q $-isogeny that is unconditional on the Birch--Swinnerton-Dyer conjecture, which is useful for computing residual densities modulo $ q > 3 $.

After establishing notational conventions in the next section, some background on Manin's modular symbols will be provided, partly to recall known integrality results for $ \LLL\br{E, \chi} $, but mainly to obtain an explicit congruence between $ \LLL\br{E, 1} $ and $ \LLL\br{E, \chi} $ modulo $ \abr{1 - \zeta_q} $. The rest of the paper will be devoted to proving the four results, with an appendix consisting of mod-$ 3 $ and $ 3 $-adic Galois images for reference.

\pagebreak

\section{Background and conventions}

This section establishes some relevant background on Galois representations and L-functions of elliptic curves, as well as some notational conventions that might be deemed less standard in the literature.

For a primitive $ n $-th root of unity $ \zeta_n $, the ring of integers of the cyclotomic field $ \QQ\br{\zeta_n} $ will be denoted $ \ZZ\sbr{\zeta_n} $, and denote the associated norm map by $ \Nm_n : \QQ\br{\zeta_n} \to \QQ $. The ring of integers of its maximal totally real subfield $ \QQ\br{\zeta_n}^+ $ will be denoted $ \ZZ\sbr{\zeta_n}^+ $, and denote the associated norm map by $ \Nm_n^+ : \QQ\br{\zeta_n}^+ \to \QQ $. The isomorphism $ \br{\ZZ / n\ZZ}^\times \xrightarrow{\sim} \Gal\br[1]{\QQ\br{\zeta_n} / \QQ} $ will be given by $ a \mapsto \br{\zeta_n \mapsto \zeta_n^a} $, which identifies Dirichlet characters of modulus $ n $ with Artin representations that factor through $ \QQ\br{\zeta_n} $.

Denote the special linear group, the general linear group, and the projective linear group respectively by
$$ \SL\br{n} \coloneqq \SL_2\br{\ZZ / n\ZZ}, \qquad \GL\br{n} \coloneqq \GL_2\br{\ZZ / n\ZZ}, \qquad \PGL\br{n} \coloneqq \PGL_2\br{\ZZ / n\ZZ}. $$
For a matrix $ M $ in such a matrix group, its trace will be denoted $ \tr\br{M} $ and its determinant will be denoted $ \det\br{M} $. For a prime $ q $, the following table groups the conjugacy classes of $ \SL\br{q} $ by trace \cite[Table 1.1 and Exercise 1.4]{Bon11}, which will be useful for Theorem \ref{thm:valuation} and Proposition \ref{prop:density}.
$$
\begin{array}{|c|c|c|c|c|}
\hline
\text{Representative} & \text{Number of classes} & \text{Order} & \text{Cardinality} & \text{Trace} \\
\hline
\begin{array}{c} \twobytwo{1}{z}{0}{1}, \\ z \in \FF_q \end{array} & \begin{array}{c} \text{one for each symbol} \ \br{\tfrac{z}{q}}, \\ \text{for a total of} \ 3 \ \text{classes} \end{array} & q^{|(\frac{z}{q})|} & \br{\tfrac{q^2 - 1}{2}}^{|(\frac{z}{q})|} & 2 \\
\hline
\begin{array}{c} \twobytwo{-1}{z}{0}{-1}, \\ z \in \FF_q \end{array} & \begin{array}{c} \text{one for each symbol} \ \br{\tfrac{z}{q}}, \\ \text{for a total of} \ 3 \ \text{classes} \end{array} & 2q^{|(\frac{z}{q})|} & \br{\tfrac{q^2 - 1}{2}}^{|(\frac{z}{q})|} & q - 2 \\
\hline
\begin{array}{c} \twobytwo{x}{0}{0}{x^{-1}}, \\ x \in \FF_q^\times \setminus \cbr{\pm1} \end{array} &
\begin{array}{c} \text{one for each pair} \ \cbr{x, x^{-1}}, \\ \text{for a total of} \ \tfrac{q - 3}{2} \ \text{classes} \end{array} & \text{order of} \ x & q\br{q + 1} & x + x^{-1} \\
\hline
\begin{array}{c} \twobytwo{\tfrac{1}{2}\br{\xi + \xi^q}}{\tfrac{\zeta}{2}\br{\xi - \xi^q}}{\tfrac{1}{2\zeta}\br{\xi - \xi^q}}{\tfrac{1}{2}\br{\xi + \xi^q}}, \\ \xi \in \br[1]{\FF_{q^2}^\times / \FF_q^\times} \setminus \cbr{\pm1} \end{array} & \begin{array}{c} \text{one for each pair} \ \cbr{\xi, \xi^{-1}}, \\ \text{for a total of} \ \tfrac{q - 1}{2} \ \text{classes} \end{array} & \text{order of} \ \xi & q\br{q - 1} & \xi + \xi^q \\
\hline
\end{array}
$$
Here, $ \zeta $ is a fixed element of $ \FF_{q^2}^\times $ satisfying $ \zeta + \zeta^q = 0 $, and $ \br{\tfrac{z}{q}} $ is the Legendre symbol modulo $ q $ given by
$$ \br{\dfrac{z}{q}} \coloneqq
\begin{cases}
1 & \text{if} \ z \ \text{is a quadratic residue modulo} \ q \\
0 & \text{if} \ q \mid z \\
-1 & \text{if} \ z \ \text{is a quadratic nonresidue modulo} \ q
\end{cases}.
$$

Throughout, an \textbf{elliptic curve} will always refer to an elliptic curve $ E $ over $ \QQ $ of conductor $ N $, and any explicit example of an elliptic curve will be given by its Cremona label \cite[Table 1]{Cre92}. For a prime $ q $, the $ \ZZ_q $-representation associated to the $ q $-adic Tate module of $ E $ is denoted $ \rho_{E, q} $, and its $ q $-adic Galois image $ \im\rho_{E, q} $ will be given by its Rouse--Sutherland--Zureick-Brown label as a subgroup of $ \GL_2\br{\ZZ_q} $ up to conjugacy \cite[Section 2.4]{RSZB22}. For any $ n \in \NN $, the projection of $ \rho_{E, q} $ onto $ \GL\br{q^n} $ is denoted $ \overline{\rho_{E, q^n}} $, and its mod-$ q $ Galois image $ \im\overline{\rho_{E, q^n}} $ will be given by its Sutherland label as a subgroup of $ \GL\br{q^n} $ up to conjugacy \cite[Section 6.4]{Sut16}. Note that if $ \Fr_p $ is an arithmetic Frobenius at a prime $ p \ne q $, then
$$ \tr\br[1]{\rho_{E, q}\br{\Fr_p}} = \a_p\br{E} \coloneqq 1 + p - \#E\br{\FF_p}. $$

Let $ \omega_E $ denote a global invariant differential on a minimal Weierstrass equation of $ E $. Let $ \X_0\br{N} $ denote the modular curve associated to the congruence subgroup $ \Gamma_0\br{N} $, and let $ \S_2\br{N} $ denote the space of weight two cusp forms and level $ \Gamma_0\br{N} $. By the modularity theorem, there is a surjective morphism $ \phi_E : \X_0\br{N} \twoheadrightarrow E $ of minimal degree and an eigenform $ f_E \in \S_2\br{N} $ with Fourier coefficients $ \a_p\br{E} $ for each prime $ p $. These define two differentials on $ \X_0\br{N} $, namely $ 2\pi if_E\br{z}\d z $ and the pullback $ \phi_E^*\omega_E $, which are related by
$$ \phi_E^*\omega_E = \pm\c_0\br{E}2\pi if_E\br{z}\d z, $$
where $ \c_0\br{E} $ is a positive integer called the Manin constant \cite[Proposition 2]{Edi91}.

\pagebreak

When $ E $ is $ \Gamma_0\br{N} $-optimal in its isogeny class, it is conjectured that $ \c_0\br{E} = 1 $ always, and this was recently proven for semistable $ E $ \cite[Theorem 1.2]{C18}, but it is certainly possible that $ \c_0\br{E} \ne 1 $ in general. Nevertheless, every modular parameterisation by $ \X_0\br{N} $ factors through a modular parameterisation by the modular curve $ \X_1\br{N} $ associated to the congruence subgroup $ \Gamma_1\br{N} $ \cite[Theorem 1.9]{Ste89}, and an analogous construction using $ \X_1\br{N} $ yields the Manin constant $ \c_1\br{E} $ with the following conjecture.

\begin{conjecture}[Stevens]
Let $ E $ be an elliptic curve. Then $ \c_1\br{E} = 1 $.
\end{conjecture}

The L-function $ \L\br{E, s} $ of $ E $ is defined to be the Euler product of $ \L_p\br[1]{\rho_{E, q}^\vee, p^{-s}}^{-1} $ over all primes $ p $, where $ \rho_{E, q}^\vee $ is the dual of the Artin representation associated to $ \rho_{E, q} $ for some prime number $ q \ne p $. Here, for an Artin representation $ \rho $, the local Euler factors are given by $ \L_p\br{\rho, T} \coloneqq \det\br[1]{1 - T \cdot \Fr_p^{-1} \mid \rho^p} $, where $ \rho^p $ is the subrepresentation of $ \rho $ fixed by the inertia subgroup at $ p $. The modularity theorem says that $ \L\br{E, s} $ is the Hecke L-function of $ f_E $, so its order of vanishing at $ s = 1 $ is well-defined.

The Birch--Swinnerton-Dyer conjecture predicts this order of vanishing and its leading term in terms of arithmetic invariants as follows. Let $ \tor\br{E} $ and $ \rk\br{E} $ denote the torsion subgroup and the rank of the Mordell--Weil group $ E\br{\QQ} $ respectively. Let $ \Omega\br{E} $ denote the real period given by $ \int_{E\br{\RR}} \omega_E $, with orientation chosen such that $ \Omega\br{E} > 0 $. Let $ \Tam\br{E} $ denote the Tamagawa product of local Tamagawa numbers $ \Tam_p\br{E} $ at each prime $ p $. Let $ \Reg\br{E} $ denote the elliptic regulator defined in terms of the N\'eron--Tate pairing $ \abr{P, Q} = \tfrac{1}{2}\h_E\br{P + Q} - \tfrac{1}{2}\h_E\br{P} - \tfrac{1}{2}\h_E\br{Q} $, where $ \h_E $ is the canonical height on $ E $. Finally, let $ \Sha\br{E} $ denote the Tate--Shafarevich group, which is implicitly assumed to be finite in this paper.

\begin{conjecture}[Birch--Swinnerton-Dyer]
Let $ E $ be an elliptic curve. Then the order of vanishing $ r $ of $ \L\br{E, s} $ at $ s = 1 $ is equal to $ \rk\br{E} $, and its leading term satisfies
$$ \lim_{s \to 1} \dfrac{\L\br{E, s}}{\br{s - 1}^r} \cdot \dfrac{1}{\Omega\br{E}} = \dfrac{\Tam\br{E} \cdot \#\Sha\br{E} \cdot \Reg\br{E}}{\#\tor\br{E}^2}. $$
\end{conjecture}

Here, the left hand side is the algebraic L-value of $ E $, which will be denoted $ \LLL\br{E} $, and the right hand side is the Birch--Swinnerton-Dyer quotient of $ E $, which will be denoted $ \BSD\br{E} $. By the combined works of Gross--Zagier \cite[Theorem 7.3]{GZ86} and Kolyvagin \cite[Corollary 2]{Kol88}, it is known that $ \L\br{E, 1} \ne 0 $ implies that $ \rk\br{E} = 0 $ and $ \Sha\br{E} $ is finite. In this setting, $ \BSD\br{E} $ is clearly rational since $ \Reg\br{E} = 1 $, and Proposition \ref{prop:untwisted} shows that $ \LLL\br{E} $ is also rational. If $ \ord_q : \QQ \to \ZZ \cup \cbr{\infty} $ denotes the $ q $-adic valuation for some prime $ q $, the conjecture that $ \ord_q\br{\LLL\br{E}} = \ord_q\br{\BSD\br{E}} $ is called the $ q $-part of the Birch--Swinnerton-Dyer conjecture. For the base change $ E / K $ of $ E $ to an extension $ K $ of $ \QQ $, the analogous quantities $ \LLL\br[0]{E / K} $ and $ \BSD\br[0]{E / K} $ can be defined as in the paper by Dokchitser--Evans--Wiersema \cite[Section 1.5]{DEW21}.

Throughout, a \textbf{character} will always refer to a non-trivial even primitive Dirichlet character $ \chi $ of conductor $ n $ coprime to $ N $ and order $ q > 1 $, which means that $ \chi\br{-1} = 1 $, and furthermore $ q $ divides the Euler totient of $ n $. The L-function $ \L\br{E, \chi, s} $ of $ E $ twisted by $ \chi $ is defined to be the Euler product of $ \L_p\br[1]{\rho_{E, q}^\vee \otimes \overline{\chi}, p^{-s}}^{-1} $ over all primes $ p $, so that in particular $ \LLL\br{E, 1} = \LLL\br{E} $. The modularity theorem says that $ \L\br{E, \chi, s} $ is the Hecke L-function of a weight two cusp form of level $ \Gamma_0\br{N} $ twisted by $ \chi $ \cite[Theorem 3.66]{Shi71}, so its order of vanishing at $ s = 1 $ is again well-defined. The analogous algebraic twisted L-value is given by
$$ \LLL\br{E, \chi} \coloneqq \lim_{s \to 1} \dfrac{\L\br{E, \chi, s}}{\br{s - 1}^r} \cdot \dfrac{n}{\tau\br{\chi}\Omega\br{E}}, $$
where $ r $ is the order of vanishing of $ \L\br{E, \chi, s} $ at $ s = 1 $ and $ \tau\br{\chi} $ is the Gauss sum of $ \chi $.

\begin{remark}
The definitions of L-values and Birch--Swinnerton-Dyer invariants in this section agree with those in Wiersema--Wuthrich \cite[Section 7]{WW22} and those in Dokchitser--Evans--Wiersema \cite[Section 1.5]{DEW21}, except for one notable difference for twisted L-functions. In this paper, the Dirichlet series of $ \L\br{E, \chi, s} $ is $ \sum_{n = 1}^\infty \chi\br{n}\a_n\br{E}n^{-s} $, and $ \LLL\br{E, \chi} $ is defined in terms of $ \L\br{E, \chi, s} $. Wiersema--Wuthrich gives two definitions for twisted L-functions, namely an automorphic one that agrees with $ \L\br{E, \chi, s} $, and a motivic one that coincides with $ \L\br{E, \overline{\chi}, s} $ instead of $ \L\br{E, \chi, s} $. However, their algebraic twisted L-value is defined motivically, so it coincides with $ \LLL\br{E, \overline{\chi}} $ instead of $ \LLL\br{E, \chi} $. Dokchitser--Evans--Wiersema follows the motivic convention, so their twisted L-functions and algebraic twisted L-values coincide with $ \L\br{E, \overline{\chi}, s} $ and $ \LLL\br{E, \overline{\chi}} $ respectively.
\end{remark}

\begin{remark}
Except for generalities on modular symbols in the next section, characters will have prime conductor and prime order, but many results can be generalised to prime power orders with a bit more work.
\end{remark}

\pagebreak

\section{Modular symbols}

This section recalls some classical facts on modular symbols. Most of the arguments here are well-known since the time of Manin \cite{Man72}, with a few recent integrality results by Wiersema--Wuthrich \cite{WW22}, but they are provided for reference. Nevertheless, the main tool is the congruence in Corollary \ref{cor:congruence}.

Let $ N \in \NN $, and let $ \S_2\br{N} $ denote the space of weight two cusp forms on the congruence subgroup
$$ \Gamma_0\br{N} \coloneqq \cbr{\twobytwo{a}{b}{c}{d} \in \PSL_2\br{\ZZ} \st N \mid c}, $$
which acts on the extended upper half plane $ \HHH $ by fractional linear transformations. A smooth path between two points in the same $ \Gamma_0\br{N} $-orbit projects onto a closed path in the quotient $ \X_0\br{N} = \HHH / \Gamma_0\br{N} $, which defines an integral homology class $ \gamma \in \H_1\br{\X_0\br{N}, \ZZ} $. This is independent of the smooth path chosen because $ \HHH $ is simply connected, and any integral homology class $ \gamma \in \H_1\br{\X_0\br{N}, \ZZ} $ arises in such a way. On the other hand, any cusp form $ f \in \S_2\br{N} $ induces a differential $ 2\pi if\br{z}\d z $ on $ \X_0\br{N} $, and integrating this over the closed path $ \gamma $ gives a complex number $ \int_\gamma 2\pi if\br{z}\d z $ called a \textbf{modular symbol}. A general definition for paths with arbitrary endpoints is given by Manin \cite[Section 1.2]{Man72}, but for the purposes of this paper, it suffices to consider the modular symbol associated to the path from $ 0 $ to cusps $ q \in \QQ \cup \cbr{\infty} $. When the denominator of $ q \in \QQ $ is coprime to $ N $, the image of any smooth path between $ 0 $ and $ q $ is always closed \cite[Proposition 2.2]{Man72}, so it makes sense to write the modular symbol
$$ \mu_f\br{q} \coloneqq \int_0^q 2\pi if\br{z}\d z \in \CC. $$
The key example for $ f $ will be the normalised cuspidal eigenform $ f_E \in \S_2\br{N} $ associated to an elliptic curve $ E $ of conductor $ N $. In this case, it turns out that $ \LLL\br{E} $, as well as $ \LLL\br{E, \chi} $ for any character $ \chi $ of conductor coprime to $ N $, can be written as sums of $ \mu_E\br{q} \coloneqq \mu_{f_E}\br{q} $ for some $ q \in \QQ $. Furthermore, the terms in these sums can be paired up in a way that guarantees integrality, using the following trick.

\begin{lemma}
\label{lem:modular}
Let $ q \in \QQ $ with denominator coprime to $ N \in \NN $, and let $ f \in \S_2\br{N} $. Then
$$ \mu_f\br{q} + \mu_f\br{1 - q} = 2\Re\br{\mu_f\br{q}}. $$
In particular, if $ E $ is an elliptic curve, then $ \mu_E\br{q} + \mu_E\br{1 - q} $ is an integer multiple of $ \c_0\br{E}^{-1}\Omega\br{E} $.
\end{lemma}

\begin{proof}
This is similar to the proof in Wiersema--Wuthrich \cite[Lemma 4]{WW22}, but the argument is repeated here for reference. Note that $ \mu_f\br{1 - q} - \mu_f\br{-q} $ is the integral of $ 2\pi if\br{z} $ along the smooth path between $ -q $ and $ \twobytwosmall{1}{1}{0}{1} \cdot \br{-q} $, which is zero \cite[Proposition 1.4]{Man72}, so $ \mu_f\br{1 - q} = \mu_f\br{-q} $. The change of variables $ z \mapsto -\overline{z} $ then transforms $ \mu_f\br{-q} $ into $ \overline{\mu_f\br{q}} $, so the first statement follows. Now let $ \gamma $ be the image of any smooth path from $ 0 $ to $ q $ under the modular parameterisation $ \phi_E : \X_0\br{N} \twoheadrightarrow E $ associated to $ E $. Then $ \mu_E\br{q} = \c_0\br{E}^{-1}\int_\gamma \omega_E $ by definition, so the second statement follows from the first statement and the fact that $ \gamma $ is necessarily an integer multiple of the generator of $ \H_1\br{E\br{\CC}, \ZZ} $ that defines $ \Omega\br{E} $.
\end{proof}

\begin{remark}
When the denominator of $ q \in \QQ $ is coprime to $ N $, these modular symbols $ \mu_E\br{q} $ coincide precisely with the modular symbols $ \mu\br{q} $ defined in Wiersema--Wuthrich \cite[Section 2]{WW22}.
\end{remark}

For this exact reason, the modular symbols $ \mu_E\br{q} $ can be normalised to be integers. More precisely, for an elliptic curve $ E $ of conductor $ N $ with normalised cuspidal eigenform $ f_E \in \S_2\br{N} $, define
$$ \mu_E^+\br{q} \coloneqq \dfrac{\c_0\br{E}}{\Omega\br{E}}\br{\mu_E\br{q} + \mu_E\br{1 - q}} \in \ZZ. $$
The integrality of $ \LLL\br{E} $ is now a formal consequence of the Hecke action on the space of modular symbols.

\pagebreak

\begin{proposition}
\label{prop:untwisted}
Let $ E $ be an elliptic curve of conductor $ N $ coprime to $ n \in \NN $. Then
\vspace{-0.2cm} $$ \c_0\br{E}\LLL\br{E}\br{\a_n\br{E} - \sigma_1\br{n} + \#E\br{\FF_2}\sigma_0^+\br{n}} = \sum_{m \mid n} \sum_{a = 1}^{\fbr{\tfrac{m - 1}{2}}} \mu_E^+\br{\tfrac{a}{m}}, $$
where $ \sigma_1\br{n} $ is the sum of divisors of $ n $ and $ \sigma_0^+ $ is the number of even divisors of $ n $. In particular, both sides lie in $ \ZZ $, and if $ p \nmid N $ is an odd prime, then
\vspace{-0.3cm} $$ \c_0\br{E}\LLL\br{E}\#E\br{\FF_p} = \sum_{a = 1}^{\fbr{\tfrac{p - 1}{2}}} \mu_E^+\br{\tfrac{a}{p}}. $$
\end{proposition}

\begin{proof}
Integrality and the final statement are immediate consequences from Lemma \ref{lem:modular} and the first statement. To prove the first statement, assume that $ 2 \nmid N $. The Hecke action \cite[Theorem 4.2]{Man72} says
$$ \L\br{E, 1}\br{\a_n\br{E} - \sigma_1\br{n}} = \sum_{m \mid n} \sum_{a = 1}^{m - 1} \mu_E\br{\tfrac{a}{m}}, $$
so in particular $ \L\br{E, 1}\#E\br{\FF_2} = -\mu_E\br{\tfrac{1}{2}} $. By dividing by $ \c_0\br{E}^{-1}\Omega\br{E} $, the first statement reduces to
\vspace{-0.2cm} $$ \sum_{m \mid n} \sum_{a = 1}^{m - 1} \mu_E\br{\tfrac{a}{m}} - \mu_E\br{\tfrac{1}{2}}\sigma_0^+\br{n} = \sum_{m \mid n} \sum_{a = 1}^{\fbr{\tfrac{m - 1}{2}}} \br{\mu_E\br{\tfrac{a}{m}} + \mu_E\br{1 - \tfrac{a}{m}}}, $$
which is simply a rearrangement of sums. Now assume that $ 2 \mid N $, so that $ 2 \nmid n $ by assumption. Then the divisors $ m \mid n $ are all odd, so $ \sigma_0^+\br{n} = 0 $ and this follows again from the Hecke action.
\end{proof}

\begin{remark}
The assumption that $ N $ is coprime to $ n $ is crucial, and removing this may cause integrality to fail, such as for the elliptic curve 11a1 where $ \c_0\br{E} = 1 $ and $ \LLL\br{E} = \tfrac{1}{5} $, but $ \a_{11}\br{E} = 1 $ and $ \sigma_1\br{11} = 12 $.
\end{remark}

The same argument can be adapted for the integrality of $ \LLL\br{E, \chi} $ using Birch's formula.

\begin{proposition}
\label{prop:twisted}
Let $ E $ be an elliptic curve of conductor $ N $, and let $ \chi $ be a character of order $ q $ and conductor $ n $ coprime to $ N $. Then
\vspace{-0.3cm} $$ \c_0\br{E}\LLL\br{E, \chi} = \sum_{a = 1}^{\fbr{\tfrac{n - 1}{2}}} \overline{\chi\br{a}}\mu_E^+\br{\tfrac{a}{n}}. $$
In particular, both sides lie in $ \ZZ\sbr{\zeta_q} $. Furthermore, if $ \c_1\br{E} = 1 $, then $ \LLL\br{E, \chi} \in \ZZ\sbr{\zeta_q} $.
\end{proposition}

\begin{proof}
Integrality follows immediately from Lemma \ref{lem:modular} and the first statement. To prove the first statement, a modification of Birch's formula \cite[Lemma 6]{WW22} yields
$$ \L\br{E, \chi, 1} = \dfrac{\tau\br{\chi}}{n}\sum_{a = 1}^{n - 1} \overline{\chi\br{a}}\mu_E\br{\tfrac{a}{n}}, $$
noting that the automorphic and motivic definitions of $ \LLL\br{E, \chi} $ agree by the assumption that $ n $ is coprime to $ N $ \cite[Lemma 18]{WW22}. By dividing by $ n^{-1}\c_0\br{E}^{-1}\tau\br{\chi}\Omega\br{E} $, the first statement reduces to
\vspace{-0.2cm} $$ \sum_{a = 1}^{n - 1} \overline{\chi\br{a}}\mu_E\br{\tfrac{a}{n}} = \sum_{a = 1}^{\fbr{\tfrac{n - 1}{2}}} \overline{\chi\br{a}}\br{\mu_E\br{\tfrac{a}{n}} + \mu_E\br{1 - \tfrac{a}{n}}}. $$
Now $ \overline{\chi\br{a}} = \overline{\chi\br{n - a}} $ since $ \chi $ is even, and $ \overline{\chi\br{\tfrac{n}{2}}} = 0 $ when $ 2 \mid n $, so this is again a rearrangement of sums. The final statement is an analogous argument with $ \c_1\br{E} $ \cite[Proposition 8]{WW22}.
\end{proof}

\begin{remark}
The assumption that $ n $ is coprime to $ N $ can be weakened slightly to $ n \nmid N $ for the first two statements \cite[Proposition 7]{WW22}, and to $ n \nmid N $ for the final statement \cite[Proposition 8]{WW22}. Removing either of these conditions may cause integrality to fail, such as for the elliptic curve 50b1 satisfying $ \c_0\br{E} = 1 $ and the unique quadratic character of conductor $ 5 $, where $ \LLL\br{E, \chi} = \tfrac{1}{3} $.
\end{remark}

\pagebreak

Observe that the right hand sides of Proposition \ref{prop:untwisted} and Proposition \ref{prop:twisted} are highly similar. More precisely, since $ \overline{\chi\br{a}} \equiv 1 \mod \abr{1 - \zeta_q} $ except when $ \gcd\br{a, n} \ne 1 $, the right hand sides are congruent modulo $ \abr{1 - \zeta_q} $, up to a small error term consisting of extraneous modular symbols $ \mu_E^+\br{\tfrac{a}{m}} $ for each proper divisor $ m $ of $ n $.

\begin{corollary}
\label{cor:congruence}
Let $ E $ be an elliptic curve of conductor $ N $, and let $ \chi $ be a character of order $ q $ and conductor $ n $ coprime to $ N $. Then
$$ \c_0\br{E}\LLL\br{E, \chi} \equiv \c_0\br{E}\LLL\br{E}\br{\a_n\br{E} - \sigma_1\br{n} + \#E\br{\FF_2}\sigma_0^+\br{n}} - \epsilon_{E, n} \mod \abr{1 - \zeta_q}, $$
where
\vspace{-0.3cm} $$ \epsilon_{E, n} \coloneqq \sum_{\substack{a = 1 \\ \gcd\br{a, n} \ne 1}}^{\fbr{\tfrac{n - 1}{2}}} \mu_E^+\br{\tfrac{a}{n}} + \sum_{\substack{m \mid n \\ m \ne n}} \sum_{a = 1}^{\fbr{\tfrac{m - 1}{2}}} \mu_E^+\br{\tfrac{a}{m}}. $$
In particular, if $ p \nmid N $ is an odd prime, then $ \epsilon_{E, p} = 0 $ and
$$ \c_0\br{E}\LLL\br{E, \chi} \equiv -\c_0\br{E}\LLL\br{E}\#E\br{\FF_p} \mod \abr{1 - \zeta_q}, $$
and if furthermore $ q \nmid \c_0\br{E} $, then
$$ \LLL\br{E, \chi} \equiv -\LLL\br{E}\#E\br{\FF_p} \mod \abr{1 - \zeta_q}, $$
where the denominators of both sides are inverted modulo $ \abr{1 - \zeta_q} $.
\end{corollary}

\begin{remark}
By stating the congruence at this generality, some interesting parity results for algebraic L-values twisted by quadratic characters can also be derived. For instance, if $ \chi $ is a quadratic character of conductor $ n = p_1p_2 $, where $ p_1 $ and $ p_2 $ are two distinct primes coprime to $ N $, then
$$ \c_0\br{E}\LLL\br{E, \chi} \equiv \c_0\br{E}\LLL\br{E}\br{\a_{p_1}\br{E}\a_{p_2}\br{E} - n - p_1 - p_2 - 1} \mod 2. $$
When $ 2 \nmid N $, something similar occurs for the unique quadratic character $ \chi $ of conductor $ n = 8 $, where
$$ \c_0\br{E}\LLL\br{E, \chi} \equiv \c_0\br{E}\LLL\br{E}\br{\a_2\br{E} + 1}\br{\a_2\br{E} + 2}\br{\a_2\br{E} - 3} \mod 2. $$
In both cases, it can be verified explicitly that $ \epsilon_{E, n} \equiv 0 \mod 2 $, and the remaining terms are obtained from recurrences for Fourier coefficients of eigenforms. Neither of these will be explored further in the paper.
\end{remark}

\begin{remark}
Without having $ \c_0\br{E} $, both integrality results and the congruence easily fail in trivial ways, but $ q \nmid \c_0\br{E} $ is a relatively mild assumption, since $ \c_0\br{E} \ne 1 $ seems to be relatively rare.
\end{remark}

\begin{remark}
\label{rem:equality}
Algebraic twisted L-values $ \LLL\br{E, \chi} $ satisfy Deligne's period conjecture \cite[Theorem 2.7]{BD07}, so in particular they are Galois equivariant, in the sense that $ \LLL\br{E, \sigma \circ \chi} = \sigma\br{\LLL\br{E, \chi}} $ for any $ \sigma \in \Gal\br[1]{\QQ\br{\zeta_q} / \QQ} $. With this property, $ \LLL\br{E} $ can be expressed in terms of the sum of $ \LLL\br{E, \chi} $ for all characters $ \chi $ of a given conductor and order. For instance, when $ \chi $ is a cubic character of conductor $ n $,
$$ 1 + \chi\br{a} + \overline{\chi\br{a}} =
\begin{cases}
3 & \text{if} \ a \in \br{\ZZ / n\ZZ}^{\times3} \\
1 & \text{if} \ a \in \br{\ZZ / n\ZZ} \setminus \br{\ZZ / n\ZZ}^\times \\
0 & \text{if} \ a \in \br{\ZZ / n\ZZ}^\times \setminus \br{\ZZ / n\ZZ}^{\times3}
\end{cases},
$$
so the identities in Proposition \ref{prop:untwisted} and Proposition \ref{prop:twisted} combine to yield
\vspace{-0.3cm} $$ \c_0\br{E}\LLL\br{E, \chi} + \c_0\br{E}\LLL\br{E, \overline{\chi}} + \c_0\br{E}\LLL\br{E}\br{\a_n\br{E} - \sigma_1\br{n} + \#E\br{\FF_2}\sigma_0^+\br{n}} = 3\sum_{\substack{a = 1 \\ a \in \br[0]{\ZZ / n\ZZ}^{\times 3}}}^{\fbr{\tfrac{n - 1}{2}}} \mu_E^+\br{\tfrac{a}{n}} + \epsilon_{E, n}. $$
By Galois equivariance, the first two terms combine to $ 2\c_0\br{E}\Re\br{\LLL\br{E, \chi}} $, so this expresses $ \Re\br{\LLL\br{E, \chi}} $ in terms of $ \LLL\br{E} $ up to a few error terms consisting of modular symbols. By reducing modulo $ 3 $, this recovers the congruence in Corollary \ref{cor:congruence}, but also shows that the congruence would not a priori hold modulo $ 9 $, unless the extraneous modular symbols $ \mu_E^+\br{\tfrac{a}{n}} $ for each cubic residue $ a $ modulo $ n $ sum to a multiple of $ 3 $.
\end{remark}

\pagebreak

\section{Denominators of L-values}

\label{sec:valuation}

This section proves a few results on the $ q $-adic valuations of denominators of algebraic L-values, where $ q $ is an odd prime, which may be of independent interest. Since $ \c_0\br{E}\LLL\br{E}\#E\br{\FF_p} \in \ZZ $, the $ q $-adic valuation of $ \c_0\br{E}\LLL\br{E} \in \QQ $ can be bounded from below by the $ q $-adic valuation of $ \#E\br{\FF_p} $, which is in turn controlled by $ \tor\br{E} $ in the denominator of $ \BSD\br{E} $. When $ q \ne 3 $, under the $ q $-part of the Birch--Swinnerton-Dyer conjecture, such a lower bound follows from Lorenzini's result that $ \ord_q\br{\#\tor\br{E}} \le \ord_q\br{\Tam\br{E}} $ with finitely many exceptions \cite[Proposition 1.1]{Lor11}, but the case $ q = 3 $ requires more work.

\begin{lemma}
\label{lem:borel}
Let $ E $ be an elliptic curve without complex multiplication such that $ E $ has a point of order $ 3 $ but $ 3 \nmid \Tam\br{E} $. Then $ \im\overline{\rho_{E, 3}} $ is a Borel subgroup.
\end{lemma}

\begin{proof}
By the assumption that $ E $ has a point of order $ 3 $, $ E $ is isomorphic either to $ y^2 + cy = x^3 $ for some cube-free $ c \in \NN $, which has complex multiplication by $ \ZZ\sbr{\zeta_3} $, or to
$$ E_{1, \pm b / a} : y^2 + xy \pm \dfrac{b}{a}y = x^3, $$
for some coprime $ a, b \in \NN $ \cite[Proposition 2.4]{BR22}. If $ 3 \nmid \ord_p\br{a} $ for some prime $ p $, then $ 3 \mid \Tam_p\br[1]{E_{1, \pm b / a}} $ \cite[Theorem 3.5]{BR22}. This contradicts the assumption that $ 3 \nmid \Tam\br{E} $, so $ a = d^3 $ for some $ d \in \NN $ coprime to $ b $. The change of variables $ \br{x, y} \mapsto \br{x / d^2, y / d^3} $ yields an isomorphism from $ E_{1, \pm b / a} $ to
$$ E_{d, \pm b} : y^2 + dxy \pm by = x^3, $$
which has discriminant $ \Delta = \pm b^3\br{d^3 - 27b} $. Now let $ p \mid b $, so that $ p \mid \Delta $ and $ \ord_p\br{d^3 - 27b} = 0 $ by coprimality. By step $ 2 $ of Tate's algorithm, since $ T^2 + dT $ splits in $ \FF_p $, $ E_{d, \pm b} $ has Kodaira symbol $ \textbf{I}_{\ord_p\br{\Delta}} $ and has split mutiplicative reduction at $ p $, so $ \Tam_p\br{E_{d, \pm b}} = \ord_p\br{\Delta} = 3\ord_p\br{b} $. This forces $ b = 1 $ by the assumption that $ 3 \nmid \Tam\br{E} $, so the j-invariant of $ E_{d, \pm b} = E_{d, \pm1} $ is given by
$$ \dfrac{d^3\br{d^3 \mp 24}^3}{\pm d^3 - 27} = 27\dfrac{\br{\tfrac{27}{\pm d^3 - 27} + 1}\br{\tfrac{27}{\pm d^3 - 27} + 9}^3}{\br{\tfrac{27}{\pm d^3 - 27}}^3}, $$
which implies that $ \im\overline{\rho_{E, 3}} $ is the Borel subgroup 3B.1.1 \cite[Theorem 1.2]{Zyw15}.
\end{proof}

Assuming the $ 3 $-part of the Birch--Swinnerton-Dyer conjecture, a divisibility result for $ \BSD\br{E} $ can be derived from the integrality of $ \c_0\br{E}\LLL\br{E}\#E\br{\FF_p} $, via a case-by-case analysis on $ \im\rho_{E, 3} $.

\begin{proposition}
\label{prop:divide}
Let $ E $ be an elliptic curve of conductor $ N $ such that $ \L\br{E, 1} \ne 0 $ and $ \tor\br{E} \cong \ZZ / 3\ZZ $. Assume that $ \ord_3\br{\LLL\br{E}} \le \ord_3\br{\BSD\br{E}} $. Then $ 3 \mid \c_0\br{E}\Tam\br{E}\#\Sha\br{E} $.
\end{proposition}

\begin{proof}
Assume that $ 3 \nmid \c_0\br{E} $. By Proposition \ref{prop:untwisted} and the assumptions, for an odd prime $ p $,
$$ \ord_3\br{\dfrac{\Tam\br{E}\#\Sha\br{E}}{9}\#E\br{\FF_p}} \ge \ord_3\br{\LLL\br{E}\#E\br{\FF_p}} \ge 0, $$
so it suffices to find an odd prime $ p \nmid N $ such that $ \#E\br{\FF_p} \equiv 3 \mod 9 $. By Chebotarev's density theorem, this reduces to finding a matrix $ \M_{E, 3} \in \im\overline{\rho_{E, 9}} $ such that $ 1 + \det\br{\M_{E, 3}} - \tr\br{\M_{E, 3}} = 3 $. By inspecting Table \ref{tab:3adic}, such matrices exist for all $ \im\rho_{E, 3} $ except for the two $ 3 $-adic Galois images 9.72.0.1 and 9.72.0.5, so these have to be handled separately. If $ \im\rho_{E, 3} $ is 9.72.0.1, then $ \im\overline{\rho_{E, 3}} $ is 3Cs.1.1 and not 3B.1.1, so Lemma \ref{lem:borel} implies that $ 3 \mid \Tam\br{E} $. Otherwise $ \im\rho_{E, 3} $ is 9.72.0.5, then $ \im\overline{\rho_{E, 9}} $ fixes a subspace of the group of $ 9 $-torsion points of $ E $, so $ E\br{\QQ} \cong \ZZ / 9\ZZ $, which contradicts the assumption that $ E\br{\QQ} \cong \ZZ / 3\ZZ $.
\end{proof}

\begin{remark}
The conclusion of Proposition \ref{prop:divide} was already observed by Melistas \cite[Example 3.8]{Mel23}, where the elliptic curves 27a3, 27a4, and 54a3 all have $ E\br{\QQ} \cong \ZZ / 3\ZZ $ and $ \Tam\br{E}\#\Sha\br{E} = 1 $ but $ \c_0\br{E} = 3 $. By the work of Lorenzini, it is generally expected that the factors in $ \Tam\br{E} $ would cancel $ \#\tor\br{E} $, but in this case it is necessary to consider $ \#\Sha\br{E} $ as well, such as in the elliptic curve 1638j3 where $ E\br{\QQ} \cong \ZZ / 3\ZZ $ and $ \c_0\br{E}\Tam\br{E} = 1 $ but $ \#\Sha\br{E} = 9 $. Note also that the statement is false for $ \tor\br{E} \cong \ZZ / 3\ZZ $ but $ \rk\br{E} > 0 $, such as for the elliptic curve 91b1 where $ \c_0\br{E}\Tam\br{E}\#\Sha\br{E} = 1 $.
\end{remark}

\pagebreak

A lower bound on the $ q $-adic valuation of $ \c_0\br{E}\LLL\br{E} $ then follows under the $ q $-part of the Birch--Swinnerton-Dyer conjecture, but when $ E $ has no rational $ q $-isogeny the bound can be made unconditional.

\begin{theorem}
\label{thm:valuation}
Let $ E $ be an elliptic curve of conductor $ N $ such that $ \L\br{E, 1} \ne 0 $, and let $ q $ be an odd prime.
\begin{enumerate}
\item If $ E $ has no rational $ q $-isogeny, then $ \ord_q\br{\c_0\br{E}\LLL\br{E}} \ge 0 $.
\item Assume that $ \ord_q\br{\LLL\br{E}} = \ord_q\br{\BSD\br{E}} $. Then $ \ord_q\br{\c_0\br{E}\LLL\br{E}} \ge -1 $.
\end{enumerate}
\end{theorem}

\begin{proof}
For the first statement, by Proposition \ref{prop:untwisted}, it suffices to find an odd prime $ p \nmid N $ such that $ q \nmid \#E\br{\FF_p} $, so by Chebotarev's density theorem this reduces to finding a matrix $ M \in \im\overline{\rho_{E, q}} $ such that $ \tr\br{M} \ne 1 + \det\br{M} $. Suppose otherwise that $ \tr\br{M} = 1 + \det\br{M} $ for all $ M \in \im\overline{\rho_{E, q}} $, so in particular $ \tr\br{M} = 2 $ for all $ M \in \im\overline{\rho_{E, q}} \cap \SL\br{q} $. In this case, by inspecting the order in each conjugacy class of $ \SL\br{q} $, $ \im\overline{\rho_{E, q}} \cap \SL\br{q} $ is necessarily a $ q $-group, so in particular $ q \mid \#\im\overline{\rho_{E, q}} $. Then either $ \im\overline{\rho_{E, q}} $ is contained in a Borel subgroup of $ \GL\br{q} $ or $ \im\overline{\rho_{E, q}} $ contains $ \SL\br{q} $ \cite[Proposition 15]{Ser72}. The former contradicts the assumption that $ E $ has no rational $ q $-isogeny, and the latter is impossible by comparing orders.

For the second statement, the assumption on the $ q $-adic valuations reduces the statement to proving that $ \ord_q\br{\c_0\br{E}\BSD\br{E}} \ge -1 $. By Mazur's torsion theorem, it suffices to consider $ \tor\br{E} $ being one of
$$ \ZZ / 3\ZZ, \ \ZZ / 5\ZZ, \ \ZZ / 6\ZZ, \ \ZZ / 7\ZZ, \ \ZZ / 9\ZZ, \ \ZZ / 10\ZZ, \ \ZZ / 12\ZZ, \ \ZZ / 2\ZZ \times \ZZ / 6\ZZ, $$
since $ q $ is odd. If $ E\br{\QQ} \not\cong \ZZ / 3\ZZ $, then a case-by-case analysis of $ q $ yields $ \ord_q\br{\Tam\br{E}} \ge \ord_q\br{\#\tor\br{E}} $ except for the elliptic curve 11a3 with $ q = 5 $, and the elliptic curves 14a4 and 14a6 with $ q = 3 $ \cite[Proposition 1.1]{Lor11}, but these exceptions all have $ \ord_q\br{\c_0\br{E}} = 1 $ and $ \ord_q\br{\BSD\br{E}} = -2 $. If $ E\br{\QQ} \cong \ZZ / 3\ZZ $, then Proposition \ref{prop:divide} implies that $ \ord_3\br{\BSD\br{E}} = \ord_3\br{\c_0\br{E}\Tam\br{E}\#\Sha\br{E}} - 2 \ge -1 $.
\end{proof}

\begin{remark}
The assumption on the $ q $-part of the Birch--Swinnerton-Dyer conjecture in the second statement can be slightly weakened, by only requiring that $ \ord_q\br{\LLL\br{E}} \ge \ord_q\br{\BSD\br{E}} $ for all $ E $, except for when $ \im\rho_{E, 3} $ is 9.72.0.1, where the assumption $ \ord_q\br{\LLL\br{E}} \le \ord_q\br{\BSD\br{E}} $ is also needed to proceed with the argument in Proposition \ref{prop:divide}. The second statement might also be provable without appealing to the $ q $-part of the Birch--Swinnerton-Dyer conjecture when $ q > 3 $, by finding a matrix $ M \in \im\rho_{E, q} $ such that $ 1 + \det\br{M} - \tr\br{M} \equiv q \mod q^2 $ along the same lines as the proof of Proposition \ref{prop:divide}. In general, this would need a case-by-case analysis of $ \im\rho_{E, q} $ for when $ E $ has no rational $ q $-isogeny, which seems to be unsolved.
\end{remark}

\begin{remark}
A lower bound for the $ 2 $-adic valuation of $ \c_0\br{E}\LLL\br{E} $ should also be achievable by arguing along the same lines, using the $ n = 2 $ version of Proposition \ref{prop:untwisted} and the classification of $ \im\rho_{E, 2} $ if necessary. This will not be needed for the rest of the paper and will not be discussed further.
\end{remark}

The following is an easy result on the $ q $-adic valuation of $ \LLL\br{E}\#E\br{\FF_p} $. The factors arising from the denominator of the rational number $ \LLL\br{E}\#E\br{\FF_p} $ could a priori cancel the factors appearing in $ \c_0\br{E} $, but the congruence of L-values says that this should not happen under Stevens's conjecture that $ \c_1\br{E} = 1 $.

\begin{proposition}
\label{prop:valuation}
Let $ E $ be an elliptic curve of conductor $ N $, and let $ p \nmid N $ and $ q $ be odd primes such that $ p \equiv 1 \mod q $. Assume that $ \c_1\br{E} = 1 $. Then
$$ q \nmid \c_0\br{E}\LLL\br{E}\#E\br{\FF_p} \qquad \iff \qquad q \nmid \c_0\br{E} \ \text{and} \ \ord_q\br[1]{\LLL\br{E}\#E\br{\FF_p}} = 0. $$
\end{proposition}

\begin{proof}
Assume first that $ q \nmid \c_0\br{E}\LLL\br{E}\#E\br{\FF_p} $ but $ q \mid \c_0\br{E} $. By the assumption that $ \c_1\br{E} = 1 $, Proposition \ref{prop:twisted} says that $ \LLL\br{E, \chi} \in \ZZ\sbr{\zeta_q} $ for any character $ \chi $ of conductor $ p $ and order $ q $, so $ \c_0\br{E}\LLL\br{E, \chi} \equiv 0 \mod \abr{1 - \zeta_q} $, which contradicts $ q \nmid \c_0\br{E}\LLL\br{E}\#E\br{\FF_p} $ by Corollary \ref{cor:congruence}. Thus $ q \nmid \c_0\br{E} $, so that $ \ord_q\br[1]{\LLL\br{E}\#E\br{\FF_p}} = 0 $ also follows, while the converse is immediate noting that $ \LLL\br{E} \ne 0 $.
\end{proof}

\begin{remark}
Assuming Stevens's conjecture, Proposition \ref{prop:valuation} yields an immediate proof that $ \LLL\br{E}\#E\br{\FF_p} $ is integral at $ q $ if $ \ord_q\br{\c_0\br{E}} \le 1 $. This condition seems to hold for all elliptic curves in the LMFDB \cite{Col}, but a proof remains elusive. On the other hand, assuming the $ q $-part of the Birch--Swinnerton-Dyer conjecture, there might be a direct proof that $ \LLL\br{E}\#E\br{\FF_p} $ is integral at $ q $, by arguing that $ 1 + \det\br{M} - \tr\br{M} $ cancels $ \#\tor\br{E}^2 $ for every matrix $ M $ lying in every possible $ \im\rho_{E, q} $. Again, this will be omitted here.
\end{remark}

\pagebreak

\section{Units of twisted L-values}

\label{sec:unit}

Under standard arithmetic conjectures, Dokchitser--Evans--Wiersema computed the norm of $ \LLL\br{E, \chi} $ in terms of $ \BSD\br{E} $ and $ \BSD\br[0]{E / K} $, where $ K $ is the degree $ q $ subfield of $ \QQ\br{\zeta_p} $ cut out by the kernel of $ \chi $ \cite[Theorem 38]{DEW21}. Some of their results can be summarised in the notation of this paper as follows.

\begin{proposition}
\label{prop:bsd}
Let $ E $ be an elliptic curve of conductor $ N $, and let $ \chi $ be a character of odd prime conductor $ p \nmid N $ and odd prime order $ q \nmid \c_0\br{E}\BSD\br{E}\#E\br{\FF_p} $. Assume that $ \c_1\br{E} = 1 $, and that $ \LLL\br{E} = \BSD\br{E} $ and $ \LLL\br[0]{E / K} = \BSD\br[0]{E / K} $. Furthermore, set $ \zeta \coloneqq \chi\br{N}^{\br{q - 1} / 2} $.
\begin{enumerate}
\item $ \LLL\br{E, \chi} \in \ZZ\sbr{\zeta_q} $ generates an ideal invariant under complex conjugation, and has norm
$$ \Nm_q\br{\LLL\br{E, \chi}} = \dfrac{\BSD\br[0]{E / K}}{\BSD\br{E}}. $$
\item $ \LLL\br{E, \chi} \cdot \zeta \in \ZZ\sbr{\zeta_q}^+ $, and has norm
$$ \abs{\Nm_q^+\br{\LLL\br{E, \chi} \cdot \zeta}} = \dfrac{\sqrt{\BSD\br[0]{E / K}}}{\sqrt{\BSD\br{E}}}. $$
\end{enumerate}
In particular, if $ \BSD\br{E} = \BSD\br[0]{E / K} $, then there is a unit $ u \in \ZZ\sbr{\zeta_q}^+ $ such that $ \LLL\br{E, \chi} = u \cdot \zeta^{-1} $.
\end{proposition}

\begin{proof}
By Proposition \ref{prop:valuation}, under the arithmetic conjectures, the assumption that $ q \nmid \c_0\br{E}\BSD\br{E}\#E\br{\FF_p} $ reduces to $ q \nmid \c_0\br{E} $ and $ \ord_q\br[1]{\LLL\br{E}\#E\br{\FF_p}} = 0 $. In particular $ \L\br{E, 1} \ne 0 $, and moreover $ \ord_q\br{\LLL\br{E, \chi}} = 0 $ by Corollary \ref{cor:congruence}, so $ \L\br{E, \chi, 1} \ne 0 $ as well. This verifies the assumptions of a result by Dokchitser--Evans--Wiersema \cite[Theorem 13(5) to Theorem 13(12)]{DEW21}, and is a restatement of it.
\end{proof}

Proposition \ref{prop:bsd}.1 predicts that ideal $ I $ of $ \ZZ\sbr{\zeta_q} $ generated by $ \LLL\br{E, \chi} $ factorises into a product of prime ideals whose norms are factors of $ \BSD\br[0]{E / K} / \BSD\br{E} $, and there are only finitely many of these. The resulting product ideal $ I $ is necessarily invariant under complex conjugation, which narrows down the possibilities further. The exact prime factorisation of $ I $, and hence of $ \LLL\br{E, \chi} $, can be determined from the $ \Gal\br{K / \QQ} $-module structure of $ \Sha\br[0]{E / K} $ \cite[Remark 7.4]{BC21}, but this will not be discussed here.

Assuming that $ I $ has been computed as an abstract ideal of $ \ZZ\sbr{\zeta_q} $, any generator of $ I $ is only equal to the actual value of $ \LLL\br{E, \chi} $ up to a unit $ u \in \ZZ\sbr{\zeta_q} $. Proposition \ref{prop:bsd}.2 refines this prediction slightly by adding a condition on the norm of $ \LLL\br{E, \chi} \cdot \zeta $, which determines the actual value of $ \LLL\br{E, \chi} $ up to a unit $ u \in \ZZ\sbr{\zeta_q}^+ $. In the special case of $ q = 3 $, this is still ambiguous up to a sign, since the units of $ \ZZ\sbr{\zeta_3}^+ = \ZZ $ are $ \pm1 $. Corollary \ref{cor:congruence} comes into the picture by pinning down the sign in terms of $ \#E\br{\FF_p} $.

\begin{corollary}
\label{cor:cubic}
Let $ E $ be an elliptic curve of conductor $ N $, and let $ \chi $ be a cubic character of odd prime conductor $ p \nmid N $ such that $ 3 \nmid \c_0\br{E}\BSD\br{E}\#E\br{\FF_p} $. Assume that $ \c_1\br{E} = 1 $, and that $ \LLL\br{E} = \BSD\br{E} $ and $ \LLL\br[0]{E / K} = \BSD\br[0]{E / K} $. Then
$$ \LLL\br{E, \chi} = u \cdot \overline{\chi\br{N}}\dfrac{\sqrt{\BSD\br[0]{E / K}}}{\sqrt{\BSD\br{E}}}, $$
for some sign $ u = \pm1 $, chosen such that
$$ u \equiv -\#E\br{\FF_p}\dfrac{\sqrt{\BSD\br{E}}^3}{\sqrt{\BSD\br[0]{E / K}}} \mod 3. $$
\end{corollary}

This follows immediately from Corollary \ref{cor:congruence} and Proposition \ref{prop:bsd}. Corollary \ref{cor:cubic} clarifies much of the phenomena observed by Dokchitser--Evans--Wiersema \cite[Example 45]{DEW21}, where they gave many pairs of examples of arithmetically similar elliptic curves $ E_1 $ and $ E_2 $ with $ \LLL\br{E_1, \chi} \ne \LLL\br{E_2, \chi} $ for a few cubic characters $ \chi $, in the sense that $ \LLL\br{E_1, \chi} \ne \LLL\br{E_2, \chi} $ precisely because $ \#E_1\br{\FF_p} \not\equiv \#E_2\br{\FF_p} \mod 3 $.

\pagebreak

\begin{example}
Let $ E_1 $ and $ E_2 $ be the elliptic curves 1356d1 and 1356f1 respectively, and let $ \chi $ be the cubic character of conductor $ 7 $ such that $ \chi\br{3} = \zeta_3^2 $. Then $ \c_0\br{E_i} = \BSD\br{E_i} = \BSD\br{E_i / K} = 1 $ for $ i = 1, 2 $, so Proposition \ref{prop:bsd} implies that $ \LLL\br{E_i, \chi} = \pm\overline{\chi\br{1356}} = \pm\zeta_3^2 $, but it was a priori unclear why $ \LLL\br{E_1, \chi} = \zeta_3^2 $ and $ \LLL\br{E_2, \chi} = -\zeta_3^2 $. Corollary \ref{cor:cubic} explains this by requiring that this sign agrees with $ -\#E_i\br{\FF_7} $ modulo $ 3 $, and in this case $ \#E_1\br{\FF_7} = 11 $ and $ \#E_2\br{\FF_7} = 7 $, which are distinct modulo $ 3 $.

They provided many other examples satisfying $ \c_0\br{E} = \BSD\br{E} = \BSD\br[0]{E / K} = 1 $ with different $ \LLL\br{E, \chi} $ for a few different cubic characters $ \chi $, and they can all be explained similarly. For reference and comparison, the values of $ \LLL\br{E, \chi} $ for the above character and of $ \#E\br{\FF_7} $ are tabulated as follows.

\vspace{0.2cm}

\begin{tabular}{|c|cc|cc|cc|ccc|}
\hline
$ E $ & 1356d1 & 1356f1 & 3264r1 & 3264s1 & 3540a1 & 3540b1 & 4800i1 & 4800bj1 & 4800bm1 \\
\hline
$ \LLL\br{E, \chi} $ & $ \zeta_3^2 $ & $ -\zeta_3^2 $ & $ -\zeta_3^2 $ & $ \zeta_3^2 $ & $ -\zeta_3^2 $ & $ \zeta_3^2 $ & $ -\zeta_3^2 $ & $ -\zeta_3^2 $ & $ \zeta_3^2 $ \\
\hline
$ \#E\br{\FF_7} $ & 11 & 7 & 10 & 8 & 7 & 11 & 7 & 7 & 11 \\
\hline
\end{tabular}
\end{example}

When $ q \ne 3 $ but $ \BSD\br{E} = \BSD\br[0]{E / K} $, Proposition \ref{prop:bsd} says that $ \LLL\br{E, \chi} $ is a unit in $ \ZZ\sbr{\zeta_q} $, and Corollary \ref{cor:congruence} places an explicit congruence on this unit in terms of $ \#E\br{\FF_p} $.

\begin{corollary}
\label{cor:unit}
Let $ E $ be an elliptic curve of conductor $ N $, and let $ \chi $ be a character of odd prime conductor $ p \nmid N $ and odd prime order $ q \nmid \c_0\br{E}\BSD\br{E}\#E\br{\FF_p} $ such that $ \BSD\br{E} = \BSD\br[0]{E / K} $. Assume that $ \c_1\br{E} = 1 $, and that $ \LLL\br{E} = \BSD\br{E} $ and $ \LLL\br[0]{E / K} = \BSD\br[0]{E / K} $. Then $ \LLL\br{E, \chi} = u $ for some unit $ u \in \ZZ\sbr{\zeta_q} $, chosen such that $ u \equiv -\#E\br{\FF_p}\BSD\br{E} \mod \abr{1 - \zeta_q} $.
\end{corollary}

Again, this follows immediately from Corollary \ref{cor:congruence} and Proposition \ref{prop:bsd}. Corollary \ref{cor:unit} partially explains the remaining phenomena observed by Dokchitser--Evans--Wiersema \cite[Example 44]{DEW21}, where they gave many pairs of examples of arithmetically trivial elliptic curves $ E_1 $ and $ E_2 $ with $ \LLL\br{E_1, \chi} \ne \LLL\br{E_2, \chi} $ for quintic characters $ \chi $, in the sense that $ \LLL\br{E_1, \chi} \ne \LLL\br{E_2, \chi} $ precisely because $ \#E_1\br{\FF_p} \not\equiv \#E_2\br{\FF_p} \mod \abr{1 - \zeta_5} $.

\begin{example}
Let $ E_1 $ and $ E_2 $ be the elliptic curves 307a1 and 307c1 respectively, and let $ \chi $ be the quintic character of conductor $ 11 $ such that $ \chi\br{2} = \zeta_5 $. Then $ \c_0\br{E_i} = \BSD\br{E_i} = \BSD\br{E_i / K} = 1 $ for $ i = 1, 2 $, so Proposition \ref{prop:bsd} implies that $ \LLL\br{E_i, \chi} $ is a unit, but it was a priori unclear why $ \LLL\br{E_1, \chi} = 1 $ and $ \LLL\br{E_2, \chi} = \zeta_5u^2 $, where $ u \coloneqq 1 + \zeta_5^4 $. Corollary \ref{cor:unit} explains this by requiring that $ \LLL\br{E_i, \chi} \equiv -\#E_i\br{\FF_{11}} \mod \abr{1 - \zeta_5} $, and in this case $ \#E_1\br{\FF_{11}} = 9 $ and $ \#E_2\br{\FF_{11}} = 16 $, which are distinct modulo $ 5 $.

They provided a few other examples satisfying $ \c_0\br{E} = \BSD\br{E} = \BSD\br[0]{E / K} = 1 $ with different $ \LLL\br{E, \chi} $ for this character, and they can all be explained similarly. For reference and comparison, the values of $ \LLL\br{E, \chi} $ for the above character and of $ \#E\br{\FF_{11}} $ are tabulated as follows.

\vspace{0.2cm}

\begin{tabular}{|c|cc|cc|cc|cc|cc|}
\hline
$ E $ & 307a1 & 307c1 & 432g1 & 432h1 & 714b1 & 714h1 & 1187a1 & 1187b1 & 1216g1 & 1216k1 \\
\hline
$ \LLL\br{E, \chi} $ & $ 1 $ & $ \zeta_5u^2 $ & $ u^2 $ & $ -\zeta_5u^{-1} $ & $ 1 $ & $ -\zeta_5^4u^3 $ & $ \zeta_5^2u^{-1} $ & $ \zeta_5u^{-3} $ & $ -\zeta_5^3u^2 $ & $ \zeta_5^4u^{-1} $ \\
\hline
$ \#E\br{\FF_{11}} $ & 9 & 16 & 16 & 8 & 9 & 13 & 17 & 8 & 9 & 7 \\
\hline
\end{tabular}
\end{example}

When $ q \ne 3 $ and $ \BSD\br{E} \ne \BSD\br[0]{E / K} $, it is slightly awkward to rephrase Proposition \ref{prop:bsd} in a way that is applicable by Corollary \ref{cor:congruence}, so it is best illustrated with an example \cite[Example 46]{DEW21}.

\begin{example}
Let $ E_1 $ and $ E_2 $ be the elliptic curves 291d1 and 139a1 respectively, and let $ \chi $ be the quintic character of conductor $ 31 $ such that $ \chi\br{3} = \zeta_5^3 $. Then $ \c_0\br{E_i} = \BSD\br{E_i} = 1 $, but $ \BSD\br{E_i / K} = 11^2 $ for $ i = 1, 2 $, so Proposition \ref{prop:bsd} implies that $ \LLL\br{E_i, \chi} $ generates an ideal of norm $ 11^2 $ invariant under complex conjugation. By considering the primes above $ 11^2 $ in $ \ZZ\sbr{\zeta_5} $, there are only two such ideals, abstractly generated by $ \lambda_1 \coloneqq 3\zeta_5^3 + \zeta_5^2 + 3\zeta_5 \equiv 2 \mod \abr{1 - \zeta_5} $ and $ \lambda_2 \coloneqq \zeta_5^3 + 3\zeta_5 + 3 \equiv 2 \mod \abr{1 - \zeta_5} $, and in fact $ \abr{\LLL\br{E_i, \chi}} = \abr{\lambda_i} $. Assuming this, Corollary \ref{cor:congruence} then predicts that $ \LLL\br{E_i, \chi} = u_i\lambda_i $ for some units $ u_i \in \ZZ\sbr{\zeta_5} $ such that $ 2u_i \equiv -\#E_i\br{\FF_{31}} \mod \abr{1 - \zeta_5} $, and in this case $ \#E_1\br{\FF_{31}} = 33 \equiv 3 \mod 5 $ and $ \#E_1\br{\FF_{31}} = 23 \equiv 3 \mod 5 $, so $ u_i \equiv 1 \mod \abr{1 - \zeta_5} $. In fact, $ u_1 = \zeta_5^4 $ and $ u_2 = \zeta_5 - \zeta_5 + 1 $.

They have a few other examples with this character, and in each case the unit can be predicted modulo $ \abr{1 - \zeta_5} $ in the same fashion, but they will be omitted here for brevity.
\end{example}

\begin{remark}
As this example highlights, in general it is possible for $ \LLL\br{E_1, \chi} \equiv \LLL\br{E_2, \chi} \mod \abr{1 - \zeta_q} $ but $ \LLL\br{E_1, \chi} \ne \LLL\br{E_2, \chi} $, even when $ \c_0\br{E_i} = \BSD\br{E_i} = 1 $. There are also examples for when $ E_i $ have the same conductor and discriminant, and furthermore $ \BSD\br{E_i / K} = 1 $, such as for the elliptic curves 544b1 and 544f1 and the quintic character $ \chi $ of conductor $ 11 $ given by $ \chi\br{2} = \zeta_5 $, where $ \LLL\br{E_1, \chi} = -\zeta_5^3 - \zeta_5 $ and $ \LLL\br{E_2, \chi} = -2\zeta_5^3 - 3\zeta_5^2 - 2\zeta_5 $. This is the pair of elliptic curves with the smallest conductor satisfying the aforementioned properties but $ \LLL\br{E_1, \chi} \ne \LLL\br{E_2, \chi} $, and other examples do seem to be quite rare.
\end{remark}

\pagebreak

\section{Residual densities of twisted L-values}

\label{sec:density}

For a fixed elliptic curve $ E $ of conductor $ N $, a natural problem is to determine the asymptotic distribution of $ \LLL\br{E, \chi} $ as $ \chi $ varies over characters of some fixed prime order $ q $ but arbitrarily high odd prime conductor $ p \nmid N $. However, for each such $ p $, there are $ q - 1 $ characters $ \chi $ of conductor $ p $ and order $ q $, giving rise to $ q - 1 $ conjugates of $ \LLL\br{E, \chi} $, so a uniform choice of $ \chi $ for each $ p $ has to be made for any meaningful analysis. One solution is to observe that the residue class of $ \LLL\br{E, \chi} $ modulo $ \abr{1 - \zeta_q} $ is independent of the choice of $ \chi $, so a simpler problem would be to determine the asymptotic distribution of these residue classes instead. As in the introduction, let $ X_{E, q}^{< n} $ be the set of equivalence classes of characters of odd order $ q $ and odd prime conductor $ p \nmid N $ less than $ n $, where two characters in $ X_{E, q}^{< n} $ are equivalent if they have the same conductor. Define the residual densities $ \delta_{E, q} $ of $ \LLL\br{E, \chi} $ to be the natural densities of $ \LLL\br{E, \chi} $ modulo $ \abr{1 - \zeta_q} $, namely
$$ \delta_{E, q}\br{\lambda} \coloneqq \lim_{n \to \infty} \dfrac{\#\cbr{\chi \in X_{E, q}^{< n} \st \LLL\br{E, \chi} \equiv \lambda \mod \abr{1 - \zeta_q}}}{\#X_{E, q}^{< n}}, \qquad \lambda \in \FF_q, $$
if such a limit exists. When $ q \nmid \c_0\br{E} $, this can be computed for each $ \lambda \in \FF_q $ directly using Corollary \ref{cor:congruence}, with the only subtlety being the possible cancellations between $ \LLL\br{E} $ and $ \#E\br{\FF_p} $. In the generic scenario when $ \im\overline{\rho_{E, q}} $ is maximal, there is a clean description in terms of Legendre symbols.

\begin{proposition}
\label{prop:density}
Let $ E $ be an elliptic curve such that $ \L\br{E, 1} \ne 0 $, and let $ q \nmid \c_0\br{E} $ be an odd prime.
\begin{enumerate}
\item If $ \ord_q\br{\LLL\br{E}} > 0 $, then $ \delta_{E, q}\br{0} = 1 $ and $ \delta_{E, q}\br{\lambda} = 0 $ for any $ \lambda \in \FF_q^\times $.
\item If $ \ord_q\br{\LLL\br{E}} \le 0 $, then set $ m \coloneqq 1 - \ord_q\br{\LLL\br{E}} $ and
$$ \G_{E, q^m} \coloneqq \cbr{M \in \im\overline{\rho_{E, q^m}} \st \det\br{M} \equiv 1 \mod q}. $$
Then for any $ \lambda \in \FF_q $,
$$ \delta_{E, q}\br{\lambda} = \dfrac{\#\cbr{M \in \G_{E, q^m} \st 1 + \det\br{M} - \tr\br{M} \equiv -\lambda\LLL\br{E}^{-1} \mod q^m}}{\#\G_{E, q^m}}. $$
\end{enumerate}
In particular, if $ \ord_q\br{\LLL\br{E}} \le 0 $ and $ \overline{\rho_{E, q}} $ is surjective, then set
$$ \lambda_{E, q} \coloneqq \br{\dfrac{\lambda\LLL\br{E}^{-1}}{q}}\br{\dfrac{\lambda\LLL\br{E}^{-1} + 4}{q}}. $$
Then for any $ \lambda \in \FF_q $,
$$ \delta_{E, q}\br{\lambda} =
\begin{cases}
\dfrac{1}{q - 1} & \text{if} \ \lambda_{E, q} = 1 \\
\dfrac{q}{q^2 - 1} & \text{if} \ \lambda_{E, q} = 0 \\
\dfrac{1}{q + 1} & \text{if} \ \lambda_{E, q} = -1
\end{cases}.
$$
\end{proposition}

\begin{proof}
By Corollary \ref{cor:congruence}, $ \delta_{E, q}\br{\lambda} $ is just the natural density of $ -\LLL\br{E}\#E\br{\FF_p} \equiv \lambda \mod q $. If $ \ord_q\br{\LLL\br{E}} > 0 $, then only $ \lambda = 0 $ gives a non-zero natural density, otherwise this is equivalent to $ 1 + p - \a_p\br{E} \equiv -\lambda\LLL\br{E}^{-1} \mod q^m $, noting that $ \LLL\br{E}^{-1} $ is well-defined and non-zero modulo $ q^m $ by definition. By Chebotarev's density theorem, this occurs with the proportion of matrices $ M \in \G_{E, q} $ with $ \det\br{M} = p $ and $ \tr\br{M} = \a_p\br{E} $, so the second statement follows. If $ \overline{\rho_{E, q}} $ is surjective, then Theorem \ref{thm:valuation}.1 yields $ m = 1 $, so $ \delta_{E, q}\br{\lambda} $ is the proportion of matrices $ M \in \SL\br{q} $ such that $ \tr\br{M} \equiv 2 - \lambda\LLL\br{E}^{-1} \mod q $. The final statement then follows by $ \#\SL\br{q} = \br{q - 1}q\br{q + 1} $ and by inspecting the trace in each conjugacy class of $ \SL\br{q} $, noting that $ \tr\br{M} = x + x^{-1} $ for some $ x \in \FF_q \setminus \cbr{\pm1} $ precisely when $ x^2 - 4 $ is a quadratic residue modulo $ q $.
\end{proof}

\begin{remark}
Without the assumption $ q \nmid \c_0\br{E} $, the same argument can be used to compute the residual density of $ \c_0\br{E}\LLL\br{E, \chi} $ instead, by adding a factor of $ \c_0\br{E} $ to every instance of $ \LLL\br{E} $ in the statement and proof of Proposition \ref{prop:density}. On the other hand, Proposition \ref{prop:twisted} predicts that $ \LLL\br{E, \chi} \in \ZZ\sbr{\zeta_q} $ under Stevens's conjecture, so both sides of the congruence are divisible by $ q $ and the statement becomes vacuous.
\end{remark}

\pagebreak

\begin{remark}
Under standard arithmetic conjectures, Proposition \ref{prop:bsd} says that $ \LLL\br{E, \chi} \cdot \zeta \in \ZZ\sbr{\zeta_q} $, so $ \Nm_q^+\br{\LLL\br{E, \chi} \cdot \zeta} \in \ZZ $. Since the norm is multiplicative and $ \zeta \equiv 1 \mod \abr{1 - \zeta_q} $, the asymptotic distribution of the residue class of $ \Nm_q^+\br{\LLL\br{E, \chi} \cdot \zeta} $ modulo $ q $ essentially boils down to computing $ \delta_{E, q} $.
\end{remark}

Assuming the $ q $-part of the Birch--Swinnerton-Dyer conjecture, Theorem \ref{thm:valuation}.3 says $ \ord_q\br{\BSD\br{E}} \ge -1 $, so non-trivial $ \delta_{E, q} $ are only visible when $ \ord_q\br{\BSD\br{E}} = 0, -1 $. Once this is determined, computing $ \delta_{E, q} $ then reduces to identifying $ \im\overline{\rho_{E, q}} $ or $ \im\overline{\rho_{E, q^2}} $, and then weighing the proportion of matrices with a certain determinant and trace. To illustrate this in action, the next result describes the possible ordered triples $ \br{\delta_{E, 3}\br{0}, \delta_{E, 3}\br{1}, \delta_{E, 3}\br{2}} $ of residual densities, which is only made possible thanks to the classification of $ 3 $-adic Galois images by Rouse--Sutherland--Zureick-Brown \cite[Corollary 1.3.1 and Corollary 12.3.3]{RSZB22}.

\begin{theorem}
\label{thm:density}
Let $ E $ be an elliptic curve such that $ 3 \nmid \c_0\br{E} $ and $ \L\br{E, 1} \ne 0 $. Assume that $ \ord_3\br{\LLL\br{E}} = \ord_3\br{\BSD\br{E}} $. Then precisely one of the following holds.
\begin{enumerate}
\item If $ \ord_3\br{\BSD\br{E}} > 0 $, then $ \delta_{E, 3}\br{0} = 1 $ and $ \delta_{E, 3}\br{1} = \delta_{E, 3}\br{2} = 0 $.
\item If $ \ord_3\br{\BSD\br{E}} = 0 $ and $ 3 \mid \#\tor\br{E} $, then $ \delta_{E, 3}\br{0} = 1 $ and $ \delta_{E, 3}\br{1} = \delta_{E, 3}\br{2} = 0 $.
\item If $ \ord_3\br{\BSD\br{E}} = 0 $ and $ 3 \nmid \#\tor\br{E} $, then $ \br{\delta_{E, 3}\br{0}, \delta_{E, 3}\br{1}, \delta_{E, 3}\br{2}} $ is given in Table \ref{tab:mod3}.
\item If $ \ord_3\br{\BSD\br{E}} = -1 $, then $ \br{\delta_{E, 3}\br{0}, \delta_{E, 3}\br{1}, \delta_{E, 3}\br{2}} $ is given in Table \ref{tab:3adic}.
\end{enumerate}
In particular, $ \br{\delta_{E, 3}\br{0}, \delta_{E, 3}\br{1}, \delta_{E, 3}\br{2}} $ only depends on $ \BSD\br{E} $ and on $ \im\overline{\rho_{E, 9}} $, and can only be one of
$$ \br{1, 0, 0}, \ \br{\tfrac{3}{8}, \tfrac{3}{8}, \tfrac{1}{4}}, \ \br{\tfrac{3}{8}, \tfrac{1}{4}, \tfrac{3}{8}}, \ \br{\tfrac{1}{2}, \tfrac{1}{2}, 0}, \ \br{\tfrac{1}{2}, 0, \tfrac{1}{2}}, \ \br{\tfrac{1}{8}, \tfrac{3}{4}, \tfrac{1}{8}}, $$
$$ \br{\tfrac{1}{8}, \tfrac{1}{8}, \tfrac{3}{4}}, \ \br{\tfrac{1}{4}, \tfrac{1}{2}, \tfrac{1}{4}}, \ \br{\tfrac{1}{4}, \tfrac{1}{4}, \tfrac{1}{2}}, \ \br{\tfrac{5}{9}, \tfrac{2}{9}, \tfrac{2}{9}}, \ \br{\tfrac{1}{3}, \tfrac{2}{3}, 0}, \ \br{\tfrac{1}{3}, 0, \tfrac{2}{3}}. $$
\end{theorem}

\begin{proof}
The fact that there are only four possibilities is an immediate consequence of Theorem \ref{thm:valuation}. By Proposition \ref{prop:density}, the first statement follows immediately under the assumption that $ \ord_3\br{\LLL\br{E}} = \ord_3\br{\BSD\br{E}} $, while the second statement follows from $ 3 \mid 1 + \det\br{M} - \tr\br{M} $ for all $ M \in \im\overline{\rho_{E, 3}} $ whenever $ 3 \mid \#\tor\br{E} $. The final statement follows from the first four, so it remains to prove the third and fourth statements.

For the third statement, it suffices to consider $ \G_{E, 3} = \im\overline{\rho_{E, 3}} \cap \SL\br{3} $, and there are only $ 5 $ possibilities for $ \im\overline{\rho_{E, 3}} $ when $ \overline{\rho_{E, 3}} $ is not surjective, as tabulated in Table \ref{tab:mod3}. If $ \overline{\rho_{E, 3}} $ is surjective, then $ \delta_{E, 3}\br{\lambda} $ is already computed in the final statement in Proposition \ref{prop:density}, while the other $ 5 $ cases are similar but easier computations. For instance, if $ \im\overline{\rho_{E, 3}} $ is 3B.1.2, then $ \G_{E, 3} $ is conjugate to the subgroup of unipotent upper triangular matrices in $ \SL\br{3} $, so counting the six matrices with each trace yields $ \delta_{E, 3}\br{0} = 1 $ and $ \delta_{E, 3}\br{1} = \delta_{E, 3}\br{2} = 0 $. Note that when $ \delta_{E, 3}\br{1} \ne \delta_{E, 3}\br{2} $, the residue of $ \BSD\br{E} $ modulo $ 3 $ would swap $ \delta_{E, 3}\br{1} $ and $ \delta_{E, 3}\br{2} $, such as in the case of $ \SL\br{3} \subseteq \im\overline{\rho_{E, 3}} $ where $ \delta_{E, 3}\br{1} = \tfrac{1}{4} $ precisely if $ \BSD\br{E} \equiv 1 \mod 3 $ and $ \delta_{E, 3}\br{1} = \tfrac{3}{8} $ otherwise.

For the fourth statement, it suffices to consider $ \G_{E, 9} $, and by the classification this is the projection onto $ \GL\br{9} $ of $ 21 $ different possible $ \im\rho_{E, 3} $, as tabulated in Table \ref{tab:3adic}. For instance, if $ \im\rho_{E, 3} $ is 3.8.0.1, then $ \G_{E, 9} $ is the preimage of the subgroup of $ \SL\br{3} $ generated by $ \twobytwosmall{1}{2}{0}{1} $ and $ \twobytwosmall{1}{2}{0}{2} $ under the canonical projection $ \GL\br{9} \twoheadrightarrow \GL\br{3} $. This preimage in $ \GL\br{9} $ consists of $ 243 $ matrices, of which $ 135 $ have trace $ 0 $ and $ 54 $ have trace $ 1 $ and $ 2 $ each, so $ \delta_{E, 3}\br{0} = \tfrac{135}{243} = \tfrac{5}{9} $ and $ \delta_{E, 3}\br{1} = \delta_{E, 3}\br{2} = \tfrac{54}{243} = \tfrac{2}{9} $. The other $ 20 $ cases are similar but easier computations, noting again the residue of $ 3\BSD\br{E} $ modulo $ 3 $ when $ \delta_{E, 3}\br{1} \ne \delta_{E, 3}\br{2} $, such as in the case of 27.648.18.1 where $ \delta_{E, 3}\br{1} = 0 $ precisely if $ 3\BSD\br{E} \equiv 2 \mod 3 $ and $ \delta_{E, 3}\br{1} = \tfrac{2}{3} $ otherwise.
\end{proof}

\begin{remark}
The first case happens when $ 3 \nmid \#\tor\br{E} $ but $ 3 \mid \Tam\br{E}\#\Sha\br{E} $, such as for the elliptic curve 50b4 where $ \BSD\br{E} = 3 $, and the second case happens when $ 9 \mid \Tam\br{E}\#\Sha\br{E} $, such as for the elliptic curve 84a1 where $ \BSD\br{E} = \tfrac{1}{2} $. It might be interesting to formulate more amenable criteria for these, but it will not be discussed here. Note also that if $ \im\overline{\rho_{E, 3}} $ is 3Cs.1.1, a much easier argument to prove that $ \delta_{E, 3}\br{0} = 1 $ and $ \delta_{E, 3}\br{1} = \delta_{E, 3}\br{2} = 0 $ is to observe that $ \G_{E, 3} $ is trivial, so $ E\br{\FF_p} $ acquires full $ 3 $-torsion for any $ p \equiv 1 \mod 3 $, and thus $ \LLL\br{E, \chi} \equiv -\BSD\br{E}\#E\br{\FF_p} \equiv 0 \mod 3 $ always.
\end{remark}

\begin{remark}
Assuming the $ 3 $-part of the Birch--Swinnerton-Dyer conjecture holds, Theorem \ref{thm:density} then describes the densities of the sign $ u $ determined in Corollary \ref{cor:cubic}, and hence the actual densities of $ \LLL\br{E, \chi} $.
\end{remark}

\pagebreak

\section{Twisted L-values of Kisilevsky--Nam}

\label{sec:kn22}

The computation of residual densities was originally motivated by the statistical data in Kisilevsky--Nam \cite[Section 7]{KN22}. They numerically computed millions of algebraic twisted L-values by fixing the elliptic curve and varying the character, but considered an alternative normalisation given by
$$ \LLL^+\br{E, \chi} \coloneqq
\begin{cases}
\LLL\br{E, \chi} & \text{if} \ \chi\br{N} = 1 \\
\LLL\br{E, \chi} \cdot \br{1 + \overline{\chi\br{N}}} & \text{if} \ \chi\br{N} \ne 1
\end{cases},
$$
in contrast to the normalisation factor $ \zeta $ in Proposition \ref{prop:bsd}. Under the implicit assumption that $ \LLL\br{E, \chi} \in \ZZ\sbr{\zeta_q} $, they showed that $ \LLL^+\br{E, \chi} \in \ZZ\sbr{\zeta_q}^+ $ \cite[Proposition 2.1]{KN22}, so that $ \Nm_q^+\br{\LLL^+\br{E, \chi}} \in \ZZ $. Fixing six elliptic curves $ E $ and five small orders $ q $, they varied the character $ \chi $ over millions of conductors $ p $, empirically determined the greatest common divisor $ \gcd_{E, q} $ of all the integers $ \Nm_q^+\br{\LLL^+\br{E, \chi}} $, and considered
$$ \widetilde{\LLL}^+\br{E, \chi} \coloneqq \dfrac{\Nm_q^+\br{\LLL^+\br{E, \chi}}}{\gcd_{E, q}}. $$

\begin{remark}
The definition of $ \widetilde{\LLL}^+\br{E, \chi} $ is equivalent to that of $ A_\chi $ defined by Kisilevsky--Nam when $ q $ is odd and $ \L\br{E, 1} \ne 0 $, since $ \chi\br{N} = -1 $ never occurs and the global root number is always $ 1 $ \cite[Section 2.2]{KN22}. Their definition of $ \LLL\br{E, \chi} $ has an extra factor of $ 2 $, but this is cancelled out after division by $ \gcd_{E, q} $.
\end{remark}

\begin{remark}
In the interpretation of Proposition \ref{prop:bsd}, the integer $ \gcd_{E, q} $ is predicted to arise from contributions by the greatest common divisors of $ \BSD\br[0]{E / K} / \BSD\br{E} $ ranging over various number fields $ K $ of degree $ q $ over $ \QQ $ coming from characters of order $ q $, but this will not be discussed here.
\end{remark}

As their normalisation differs from that in Proposition \ref{prop:bsd} \cite[Remark 1]{KN22}, the resulting residual densities are skewed. More precisely, define $ X_{E, q}^{< n} $ as before, and define the analogous residual densities $ \delta_{E, q}' $ of $ \widetilde{\LLL}^+\br{E, \chi} $ to be the natural densities of $ \widetilde{\LLL}^+\br{E, \chi} $ modulo $ q $, or in other words
$$ \delta_{E, q}\br{\lambda} \coloneqq \lim_{n \to \infty} \dfrac{\#\cbr{\chi \in X_{E, q}^{< n} \st \widetilde{\LLL}^+\br{E, \chi} \equiv \lambda \mod q}}{\#X_{E, q}^{< n}}, \qquad \lambda \in \FF_q, $$
if such a limit exists. In the simplest case where $ q = 3 $ and $ 3 \nmid \gcd_{E, 3} $, there is no norm, and so
$$ \widetilde{\LLL}^+\br{E, \chi} \equiv
\begin{cases}
\LLL\br{E, \chi}\gcd_{E, 3} & \text{if} \ \chi\br{N} = 1 \\
2\LLL\br{E, \chi}\gcd_{E, 3} & \text{otherwise}
\end{cases},
$$
which becomes amenable to a similar computation to that of Proposition \ref{prop:density} provided $ \chi\br{N} $ is known. For certain elliptic curves, $ \chi\br{N} $ depends completely on $ \#E\br{\FF_p} $ due to a shared action of Frobenius in $ \GL\br{3} $.

\begin{lemma}
\label{lem:cubic}
Let $ E $ be an elliptic curve of conductor $ N $ with no rational $ 3 $-isogeny such that the splitting field $ F $ of $ X^3 - N $ lies in the splitting field $ K $ of the $ 3 $-division polynomial $ \psi_{E, 3} $, and let $ \chi $ be a cubic character of odd prime conductor $ p \nmid N $. Then $ \im\overline{\rho_{E, 3}} = \GL\br{3} $ and $ \Gal\br{K / \QQ} \cong \PGL\br{3} $. Furthermore, if $ p $ does not split completely in $ K $, then $ \#E\br{\FF_p} \equiv 2 \mod 3 $ if and only if $ \chi\br{N} = 1 $. Otherwise, if $ p $ splits completely in $ K $, then $ \#E\br{\FF_p} \not\equiv 2 \mod 3 $ and $ \chi\br{N} = 1 $.
\end{lemma}

\begin{proof}
Let $ L $ be the extension of $ K $ obtained by adjoining the $ Y $-coordinates of $ E\sbr{3} $. By the assumption that $ E $ has no rational $ 3 $-isogeny and the classification of $ \im\overline{\rho_{E, 3}} $, if $ \overline{\rho_{E, 3}} $ were not surjective, then $ \Gal\br{L / \QQ} $ is either 3Nn or 3Ns. Neither of this could occur, since by the assumption that $ F \subseteq K $, there are inclusions
$$ \QQ \subseteq \QQ\br{\zeta_3} \subseteq F \subseteq K \subseteq L, $$
so in particular $ \Gal\br{L / \QQ} $ surjects onto $ \Gal\br{F / \QQ} \cong \SSS_3 $, which forces $ \Gal\br{L / \QQ} \cong \GL\br{3} $. On the other hand, $ \Gal\br{K / \QQ} $ permutes the roots of the degree $ 4 $ polynomial $ \psi_{E, 3} $, so it must be the quotient group $ \PGL\br{3} \cong \SSS_4 $. Its subgroup $ \Gal\br{K / \QQ\br{\zeta_3}} \cong \AAA_4 $ surjects onto $ \Gal\br{F / \QQ\br{\zeta_3}} \cong \ZZ / 3\ZZ $, with kernel the unique subgroup $ \Gal\br{K / F} \cong \br{\ZZ / 2\ZZ}^2 $ of index $ 4 $ consisting precisely of all elements of $ \AAA_4 $ of order $ 1 $ or $ 2 $.

\pagebreak

Now $ \Fr_p \in \Gal\br{K / \QQ} $ acts on the residue field of a prime $ \pi $ of $ F $ above $ p $ by
$$ \Fr_p\br{\zeta_3} \equiv \zeta_3^p \mod \pi, \qquad \Fr_p\br[1]{\sqrt[3]{N}} \equiv \sqrt[3]{N}^p \mod \pi. $$
Clearly $ \Fr_p $ fixes $ \zeta_3 $, so that $ \Fr_p \in \Gal\br{K / \QQ\br{\zeta_3}} $. Claim that, when $ p $ does not split completely in $ K $, the condition $ \Fr_p \in \Gal\br{K / F} $ is equivalent to $ \#E\br{\FF_p} \equiv 2 \mod 3 $ and to $ \chi\br{N} = 1 $. On one hand, this means that $ \Fr_p $ fixes $ \sqrt[3]{N} $, or equivalently that $ \sqrt[3]{N}^{p - 1} \equiv 1 \mod p $, which is precisely the condition that $ \chi\br{N} = 1 $. On the other hand, this also means that $ \Fr_p^2 = 1 $ in $ \Gal\br{K / \QQ\br{\zeta_3}} $, which is equivalent to $ \Fr_p $ having order exactly $ 2 $ in $ \Gal\br{K / \QQ\br{\zeta_3}} $. By the Cayley--Hamilton theorem, these are precisely the trace $ 0 $ matrices in $ \PGL\br{3} $, or equivalently the trace $ 0 $ matrices in $ \GL\br{3} $, which proves the equivalence with $ \a_p\br{E} = 0 $. If $ p $ splits completely in $ K $, then $ \Fr_p = 1 $ in $ \Gal\br{K / \QQ\br{\zeta_3}} $, but these never have trace $ 0 $ in $ \PGL\br{3} $ or $ \GL\br{3} $.
\end{proof}

\begin{remark}
The first assumption is necessary, evident in the elliptic curve 50b1 with $ F \subseteq K $ but $ \im\overline{\rho_{E, 3}} $ is 3B, where $ 7 $ does not split completely in $ K $ but $ \#E\br{\FF_7} = 10 \equiv 1 \mod 3 $ and $ \chi\br{50} = \overline{\chi\br{50}} = 1 $. The second assumption is also necessary, evident in the elliptic curve 21a1 with no rational $ 3 $-isogeny but $ F \not\subseteq K $, where $ 13 $ does not split completely in $ K $ but $ \#E\br{\FF_{13}} = 16 \equiv 1 \mod 3 $ and $ \chi\br{21} = \overline{\chi\br{21}} = 1 $. For the final statement, checking that $ p $ splits completely in $ F $ but not in $ K $ is not sufficient to conclude, such as for the elliptic curve 11a1, where $ \#E\br{\FF_{19}} = 20 \equiv 2 \mod 3 $ and $ \chi\br{11} = \overline{\chi\br{11}} = 1 $. If $ p $ does split completely in $ K $, then both $ \#E\br{\FF_p} \equiv 0 \mod 3 $ and $ \#E\br{\FF_p} \equiv 1 \mod 3 $ are possible, such as for the elliptic curve 11a1, where $ \#E\br{\FF_{337}} = 360 \equiv 0 \mod 3 $ and $ \#E\br{\FF_{193}} = 190 \equiv 1 \mod 3 $. Finally, note that this argument only works for cubic characters, as $ \PGL\br{q} $ is almost simple for $ q > 3 $ and admits few non-trivial surjections.
\end{remark}

For elliptic curves satisfying this property, the residual density of $ \widetilde{\LLL}^+\br{E, \chi} $ is easy to compute.

\begin{proposition}
\label{prop:cubic}
Let $ E $ be an elliptic curve of conductor $ N $ with no rational $ 3 $-isogeny such that $ 3 \nmid \c_0\br{E} $ and $ 3 \nmid \gcd_{E, 3} $, and the splitting field $ F $ of $ X^3 - N $ lies in the splitting field $ K $ of the $ 3 $-division polynomial $ \psi_{E, 3} $, and let $ \chi $ be a cubic character of odd prime conductor $ p \nmid N $. Then
$$ \widetilde{\LLL}^+\br{E, \chi} \equiv
\begin{cases}
0 \mod 3 & \text{if} \ \#E\br{\FF_p} \equiv 0 \mod 3 \\
2 \mod 3 & \text{if} \ \#E\br{\FF_p} \equiv 1 \mod 3 \ \text{and} \ p \ \text{splits completely in} \ K \\
1 \mod 3 & \text{otherwise}
\end{cases}.
$$
In particular,
$$ \delta_{E, 3}'\br{0} = \dfrac{9}{24}, \qquad \delta_{E, 3}'\br{1} = \dfrac{15}{24}, \qquad \delta_{E, 3}'\br{2} = \dfrac{1}{24}. $$
\end{proposition}

\begin{proof}
By Corollary \ref{cor:congruence} and the assumptions that $ 3 \nmid \c_0\br{E} $ and $ 3 \nmid \gcd_{E, 3} $,
$$ \widetilde{\LLL}^+\br{E, \chi} \equiv
\begin{cases}
2\#E\br{\FF_p}\LLL\br{E}\gcd_{E, 3} & \text{if} \ \chi\br{N} = 1 \\
\#E\br{\FF_p}\LLL\br{E}\gcd_{E, 3} & \text{otherwise} \\
\end{cases}.
$$
Clearly $ \widetilde{\LLL}^+\br{E, \chi} \equiv 0 \mod 3 $ when $ \#E\br{\FF_p} \equiv 0 \mod 3 $. By Lemma \ref{lem:cubic}, $ \chi\br{N} = 1 $ occurs either when $ \#E\br{\FF_p} \equiv 1 \mod 3 $ but $ p $ splits completely in $ K $ or when $ \#E\br{\FF_p} \equiv 2 \mod 3 $ but $ p $ does not split completely in $ K $, the only remaining case being when $ \#E\br{\FF_p} \equiv 1 \mod 3 $ and $ \chi\br{N} \ne 1 $. The first statement then follows by substituting the residues of $ \#E\br{\FF_p} $ modulo $ 3 $, and noting that $ \gcd_{E, 3} $ cancels out the factors in $ \LLL\br{E} $ by definition. For the final statement, the description of the groups in Lemma \ref{lem:cubic} implies that $ \#E\br{\FF_p} \equiv \lambda \mod 3 $ occurs with the proportion of matrices $ M \in \SL\br{3} $ with $ \tr\br{M} = 2 - \lambda $, by Chebotarev's density theorem. If $ p $ splits completely in $ K $, then $ \Fr_p = 1 $ in $ \PGL\br{3} $, so $ \Fr_p = \pm1 $ in $ \GL\br{3} $, and in particular in $ \SL\br{3} $, but the condition $ \#E\br{\FF_p} \equiv 1 \mod 3 $ forces $ \Fr_p = -1 $, which has trace $ 1 $. The final statement then follows by counting matrices in $ \SL\br{3} $ with a certain trace.
\end{proof}

\begin{remark}
Elliptic curves with discriminant $ \Delta = \pm N^n $ for some $ 3 \nmid n $ satisfy this property, since $ \sqrt[3]{N} $ can then be expressed in terms of $ \sqrt[3]{\Delta} $ \cite[Section 5.3b]{Ser72}. This completely explains the numerical data by Kisilevsky--Nam for the elliptic curve 11a1 where $ \gcd_{E, 3} = 5 $ and the elliptic curves 15a1 and 17a1 where $ \gcd_{E, 3} = 4 $, all of which satisfy the assumptions of Proposition \ref{prop:cubic}. The same method cannot explain the density patterns when $ 3 \mid \gcd_{E, 3} $, such as for the remaining three elliptic curves 14a1, 19a1, and 37b1 considered by Kisilevsky--Nam, since Corollary \ref{cor:congruence} is a priori not valid modulo $ 9 $, as noted in Remark \ref{rem:equality}.
\end{remark}

\pagebreak

\appendix

\section{Tables of Galois images}

\def\arraystretch{1.5}

This section tabulates the mod-$ 3 $ and $ 3 $-adic Galois images of elliptic curves $ E $ with restricted $ 3 $-torsion up to conjugacy, crucially used in Proposition \ref{prop:divide} and Theorem \ref{thm:density}. In both tables, the examples of elliptic curves are chosen so that it has the smallest conductor possible satisfying $ 3 \nmid \c_0\br{E} $ and $ \L\br{E, 1} \ne 0 $, but in general there are many elliptic curves with each prescribed mod-$ 3 $ or $ 3 $-adic Galois image.

\subsection*{Mod-3 Galois images of elliptic curves without 3-torsion}

\def\arraystretch{1}

The possible mod-$ 3 $ Galois images are well-known \cite[Theorem 1.2, Proposition 1.14, and Proposition 1.16]{Zyw15}, and those of elliptic curves without $ 3 $-torsion are tabulated as follows. The subgroup generators are taken from Sutherland \cite[Section 6.4]{Sut16}, and are viewed as elements of $ \GL\br{3} $. The final two columns give examples of elliptic curves with the given mod-$ 3 $ Galois image with $ b = 1 $ and $ b = 2 $ respectively, where $ b \in \FF_3 $ is the residue of $ \BSD\br{E} $ modulo $ 3 $. The column labelled $ \G_{E, 3} $ lists the elements of $ \G_{E, 3} = \im\overline{\rho_{E, 3}} \cap \SL\br{3} $ as defined in Proposition \ref{prop:density}, so the residual densities can be read off directly in the column labelled $ \delta_{E, 3} $ as ordered triples $ \br{\delta_{E, 3}\br{0}, \delta_{E, 3}\br{-b}, \delta_{E, 3}\br{b}} $, which is used in Theorem \ref{thm:density}.

\captionof{table}{\label{tab:mod3} mod-3 Galois images of elliptic curves without 3-torsion}

\vspace{-0.5cm}

$$
\begin{array}{|c|c|c|c|c|c|}
\hline
\im\overline{\rho_{E, 3}} & \text{Generators of} \ \im\overline{\rho_{E, 3}} & \G_{E, 3} & \delta_{E, 3} & b = 1 & b = 2 \\
\hline
\GL\br{3} & \twobytwo{2}{0}{0}{1} \twobytwo{2}{1}{2}{0} & \SL\br{3} & \tfrac{3}{8}, \tfrac{3}{8}, \tfrac{1}{4} & \text{11a2} & \text{11a1} \\
\hline
\text{3B.1.2} & \twobytwo{2}{0}{0}{1} \twobytwo{1}{1}{0}{1} & \twobytwo{1}{0}{0}{1} \twobytwo{1}{1}{0}{1} \twobytwo{1}{2}{0}{1} & 1, 0, 0 & \text{19a2} & \text{14a3} \\
\hline
\text{3B} & \twobytwo{2}{0}{0}{2} \twobytwo{1}{0}{0}{2} \twobytwo{1}{1}{0}{1} & \begin{array}{c} \twobytwo{1}{0}{0}{1} \twobytwo{1}{1}{0}{1} \twobytwo{1}{2}{0}{1} \\ \twobytwo{2}{0}{0}{2} \twobytwo{2}{1}{0}{2} \twobytwo{2}{2}{0}{2} \end{array} & \tfrac{1}{2}, \tfrac{1}{2}, 0 & \text{50b3} & \text{50b1} \\
\hline
\text{3Cs} & \twobytwo{2}{0}{0}{2} \twobytwo{1}{0}{0}{2} & \twobytwo{1}{0}{0}{1} \twobytwo{2}{0}{0}{2} & \tfrac{1}{2}, \tfrac{1}{2}, 0 & \text{304e2} & \text{304b2} \\
\hline
\text{3Nn} & \twobytwo{1}{0}{0}{2} \twobytwo{2}{1}{2}{2} & \begin{array}{c} \twobytwo{1}{0}{0}{1} \twobytwo{1}{1}{1}{2} \twobytwo{0}{2}{1}{0} \twobytwo{2}{1}{1}{1} \\ \twobytwo{2}{0}{0}{2} \twobytwo{2}{2}{2}{1} \twobytwo{0}{1}{2}{0} \twobytwo{1}{2}{2}{2} \end{array} & \tfrac{1}{8}, \tfrac{1}{8}, \tfrac{3}{4} & \text{704e1} & \text{245b1} \\
\hline
\text{3Ns} & \twobytwo{2}{0}{0}{2} \twobytwo{0}{2}{1}{0} \twobytwo{1}{0}{0}{2} & \twobytwo{1}{0}{0}{1} \twobytwo{2}{0}{0}{2} \twobytwo{0}{2}{1}{0} \twobytwo{0}{1}{2}{0} & \tfrac{1}{4}, \tfrac{1}{4}, \tfrac{1}{2} & \text{1690d1} & \text{338d1} \\
\hline
\end{array}
$$

\vspace{0.5cm}

The remaining two mod-$ 3 $ Galois images 3B.1.1 and 3Cs.1.1 have $ 3 $-torsion, so computing the residual densities require finer information from their mod-$ 9 $ Galois images.

\subsection*{3-adic Galois images of elliptic curves with 3-torsion}

The possible $ 3 $-adic Galois images are classified \cite[Corollary 1.3.1 and Corollary 12.3.3]{RSZB22}, and those of elliptic curves with $ 3 $-torsion are tabulated as follows. The subgroup generators are taken from Rouse--Sutherland--Zureick-Brown \cite[Software Repository]{RSZB22}, and are viewed as elements of $ \GL\br{3^m} $ if their corresponding $ 3 $-adic Galois images are of the form $ 3^m.i.g.n $. The column labelled $ \M_{E, 3} $ gives matrices $ \M_{E, 3} \in \im\overline{\rho_{E, 9}} $ such that $ 1 + \det\br{\M_{E, 3}} - \tr\br{\M_{E, 3}} = 3 $, which is used in Proposition \ref{prop:divide}. The final two columns give examples of elliptic curves with the given $ 3 $-adic Galois image with $ b = 1 $ and $ b = 2 $ respectively, where $ b \in \FF_3 $ is the residue of $ 3\BSD\br{E} $ modulo $ 3 $. The column labelled $ \#\G_{E, 9} $ lists the cardinalities of $ \G_{E, 9} $ as defined in Proposition \ref{prop:density} for reference, but the residual densities are calculated separately in the column labelled $ \delta_{E, 3} $ as ordered triples $ \br{\delta_{E, 3}\br{0}, \delta_{E, 3}\br{-b}, \delta_{E, 3}\br{b}} $, which is used in Theorem \ref{thm:density}.

\pagebreak

\captionof{table}{\label{tab:3adic} 3-adic Galois images of elliptic curves with 3-torsion}

\vspace{-0.5cm}

$$
\begin{array}{|c|c|c|c|c|c|c|c|}
\hline
\im\rho_{E, 3} & \im\overline{\rho_{E, 3}} & \text{Generators of} \ \im\rho_{E, 3} & \M_{E, 3} & \#\G_{E, 9} & \delta_{E, 3} & b = 1 & b = 2 \\
\hline
3.8.0.1 & \text{3B.1.1} & \twobytwo{1}{2}{0}{1} \twobytwo{1}{2}{0}{2} & \twobytwo{4}{0}{0}{2} & 243 & \tfrac{5}{9}, \tfrac{2}{9}, \tfrac{2}{9} & \text{20a2} & \text{20a1} \\
\hline
3.24.0.1 & \text{3Cs.1.1} & \twobytwo{2}{0}{0}{1} & \twobytwo{2}{0}{0}{4} & 81 & 1, 0, 0 & \text{26a1} & \text{14a1} \\
\hline
9.24.0.1 & \text{3B.1.1} & \twobytwo{7}{5}{0}{8} \twobytwo{1}{8}{0}{4} & \twobytwo{4}{0}{0}{2} & 81 & 1, 0, 0 & \text{189c3} & \text{702e3} \\
\hline
9.24.0.2 & \text{3B.1.1} & \twobytwo{7}{3}{0}{8} \twobytwo{7}{2}{6}{2} & \twobytwo{4}{0}{0}{2} & 81 & \tfrac{1}{3}, \tfrac{2}{3}, 0 & \text{} & \text{} \\
\hline
9.72.0.1 & \text{3Cs.1.1} & \twobytwo{5}{6}{3}{1} \twobytwo{4}{6}{0}{1} \twobytwo{5}{0}{0}{1} & \text{N/A} & 27 & 1, 0, 0 & \text{54b1} & \text{} \\
\hline
9.72.0.2 & \text{3Cs.1.1} & \twobytwo{8}{3}{3}{4} \twobytwo{8}{6}{0}{4} \twobytwo{1}{3}{0}{1} & \twobytwo{8}{0}{0}{4} & 27 & 1, 0, 0 & \text{54a1} & \text{} \\
\hline
9.72.0.3 & \text{3Cs.1.1} & \twobytwo{8}{3}{3}{4} \twobytwo{5}{0}{0}{7} & \twobytwo{2}{0}{0}{4} & 27 & 1, 0, 0 & \text{19a1} & \text{7094c1} \\
\hline
9.72.0.4 & \text{3Cs.1.1} & \twobytwo{2}{3}{6}{7} \twobytwo{1}{6}{6}{1} \twobytwo{4}{3}{6}{4} & \twobytwo{5}{0}{0}{4} & 27 & 1, 0, 0 & \text{} & \text{} \\
\hline
9.72.0.5 & \text{3B.1.1} & \twobytwo{1}{2}{0}{8} \twobytwo{1}{7}{0}{4} & \text{N/A} & 27 & 1, 0, 0 & \text{54b3} & \text{} \\
\hline
9.72.0.6 & \text{3B.1.1} & \twobytwo{1}{5}{0}{8} \twobytwo{4}{1}{0}{8} & \twobytwo{4}{0}{0}{8} & 27 & 1, 0, 0 & \text{} & \text{} \\
\hline
9.72.0.7 & \text{3B.1.1} & \twobytwo{4}{4}{0}{5} \twobytwo{1}{0}{0}{8} & \twobytwo{4}{0}{0}{5} & 27 & 1, 0, 0 & \text{} & \text{} \\
\hline
9.72.0.8 & \text{3B.1.1} & \twobytwo{7}{7}{6}{4} \twobytwo{7}{7}{6}{2} & \twobytwo{1}{2}{3}{1} & 27 & \tfrac{1}{3}, \tfrac{2}{3}, 0 & \text{} & \text{} \\
\hline
9.72.0.9 & \text{3B.1.1} & \twobytwo{4}{2}{3}{5} \twobytwo{1}{3}{0}{1} \twobytwo{7}{2}{3}{1} & \twobytwo{4}{1}{0}{5} & 27 & \tfrac{1}{3}, \tfrac{2}{3}, 0 & \text{} & \text{} \\
\hline
9.72.0.10 & \text{3B.1.1} & \twobytwo{1}{5}{6}{5} \twobytwo{1}{0}{0}{8} & \twobytwo{4}{0}{0}{8} & 27 & \tfrac{1}{3}, \tfrac{2}{3}, 0 & \text{486c1} & \text{} \\
\hline
27.72.0.1 & \text{3B.1.1} & \twobytwo{7}{23}{0}{5} \twobytwo{1}{8}{9}{16} & \twobytwo{4}{0}{0}{2} & 81 & 1, 0, 0 & \text{} & \text{} \\
\hline
27.648.13.25 & \text{3B.1.1} & \twobytwo{16}{4}{0}{16} \twobytwo{1}{17}{0}{26} & \twobytwo{4}{0}{0}{5} & 27 & 1, 0, 0 & \text{N/A} & \text{N/A} \\
\hline
27.648.18.1 & \text{3B.1.1} & \twobytwo{16}{15}{9}{25} \twobytwo{10}{16}{9}{17} \twobytwo{7}{22}{6}{4} & \twobytwo{4}{1}{0}{5} & 27 & \tfrac{1}{3}, \tfrac{2}{3}, 0 & \text{108a1} & \text{36a1} \\
\hline
27.1944.55.31 & \text{3Cs.1.1} & \twobytwo{2}{18}{12}{25} \twobytwo{16}{18}{21}{16} & \twobytwo{5}{0}{0}{4} & 9 & 1, 0, 0 & \text{N/A} & \text{N/A} \\
\hline
27.1944.55.37 & \text{3Cs.1.1} & \twobytwo{17}{6}{21}{10} \twobytwo{2}{3}{3}{25} & \twobytwo{5}{0}{3}{4} & 9 & 1, 0, 0 & \text{27a1} & \text{N/A} \\
\hline
27.1944.55.43 & \text{3B.1.1} & \twobytwo{19}{10}{18}{8} \twobytwo{4}{11}{3}{16} & \twobytwo{4}{4}{0}{5} & 9 & \tfrac{1}{3}, \tfrac{2}{3}, 0 & \text{243b1} & \text{N/A} \\
\hline
27.1944.55.44 & \text{3B.1.1} & \twobytwo{10}{23}{3}{13} \twobytwo{13}{13}{0}{14} & \twobytwo{4}{4}{0}{5} & 9 & \tfrac{1}{3}, \tfrac{2}{3}, 0 & \text{N/A} & \text{N/A} \\
\hline
\end{array}
$$

\vspace{0.5cm}

\begin{remark}
Note that many of the $ 3 $-adic Galois images seemingly do not represent any elliptic curves with $ b \ne 0 $, in the sense that a search through the LMFDB \cite{Col} yields no examples satisfying $ 3 \nmid \c_0\br{E} $ and $ \L\br{E, 1} \ne 0 $, but current results a priori do not rule out their existence. To rule out examples for a specific $ 3 $-adic Galois image, one could consider the explicit family of Weierstrass equations parameterised by the associated modular curve, and then investigate the divisibility of Tamagawa numbers as in Lemma \ref{lem:borel}. Such is the case for the last six $ 3 $-adic Galois images arising from elliptic curves with complex multiplication, where their associated modular curves have effectively computable finite sets of rational points.
\end{remark}

\pagebreak

\bibliography{main}

\begin{thebibliography}{\v{C}18}

\bibitem[BC21]{BC21}
David Burns and Daniel~Macias Castillo.
\newblock {On refined conjectures of Birch and Swinnerton-Dyer type for
  Hasse--Weil--Artin L-series}, 2021.

\bibitem[BD07]{BD07}
Thanasis Bouganis and Vladimir Dokchitser.
\newblock {Algebraicity of $ L $-values for elliptic curves in a false Tate
  curve tower}.
\newblock {\em Mathematical Proceedings of the Cambridge Philosophical
  Society}, 142(2):193--204, 2007.

\bibitem[Bon11]{Bon11}
C\'edric Bonnaf\'e.
\newblock {\em {Representations of $ \mathrm{SL}_2(\mathbb{F}_q) $}}, volume~13
  of {\em Algebra and Applications}.
\newblock Springer-Verlag London, Ltd., London, 2011.

\bibitem[BR22]{BR22}
Alexander Barrios and Manami Roy.
\newblock {Local data of rational elliptic curves with nontrivial torsion}.
\newblock {\em Pacific Journal of Mathematics}, 318(1):1--42, 2022.

\bibitem[Col]{Col}
The~LMFDB Collaboration.
\newblock {The L-functions and modular forms database}.

\bibitem[Cre92]{Cre92}
John Cremona.
\newblock {\em {Algorithms for modular elliptic curves}}.
\newblock Cambridge University Press, Cambridge, 1992.

\bibitem[DEW21]{DEW21}
Vladimir Dokchitser, Robert Evans, and Hanneke Wiersema.
\newblock {On a BSD-type formula for $ L $-values of Artin twists of elliptic
  curves}.
\newblock {\em Journal f\"ur die Reine und Angewandte Mathematik. [Crelle's
  Journal]}, 773:199--230, 2021.

\bibitem[Edi91]{Edi91}
Bas Edixhoven.
\newblock {On the Manin constants of modular elliptic curves}.
\newblock In {\em {Arithmetic algebraic geometry (Texel, 1989)}}, volume~89 of
  {\em Progr. Math.}, pages 25--39. Birkh\"auser Boston, Boston, MA, 1991.

\bibitem[GZ86]{GZ86}
Benedict Gross and Don Zagier.
\newblock {Heegner points and derivatives of $ L $-series}.
\newblock {\em Inventiones Mathematicae}, 84(2):225--320, 1986.

\bibitem[KN22]{KN22}
Hershy Kisilevsky and Jungbae Nam.
\newblock {Small algebraic central values of twists of elliptic $ L
  $-functions}, 2022.

\bibitem[Kol88]{Kol88}
Victor Kolyvagin.
\newblock {Finiteness of $ E(\mathbb{Q}) $ and $ \Sha(E, \mathbb{Q}) $ for a
  subclass of Weil curves}.
\newblock {\em Izvestiya Akademii Nauk SSSR. Seriya Matematicheskaya},
  52(3):522--540, 670--671, 1988.

\bibitem[Lor11]{Lor11}
Dino Lorenzini.
\newblock {Torsion and Tamagawa numbers}.
\newblock {\em Universit\'e de Grenoble. Annales de l'Institut Fourier},
  61(5):1995--2037, 2011.

\bibitem[Man72]{Man72}
Ju.~I. Manin.
\newblock {Parabolic points and zeta functions of modular curves}.
\newblock {\em Izvestiya Akademii Nauk SSSR. Seriya Matematicheskaya},
  36:19--66, 1972.

\bibitem[Mel23]{Mel23}
Mentzelos Melistas.
\newblock {A divisibility related to the Birch and Swinnerton-Dyer conjecture}.
\newblock {\em Journal of Number Theory}, 245:150--168, 2023.

\bibitem[RSZB22]{RSZB22}
Jeremy Rouse, Andrew Sutherland, and David Zureick-Brown.
\newblock {$ \ell $-adic images of Galois for elliptic curves over $ \mathbb{Q}
  $ (and an appendix with John Voight)}.
\newblock {\em Forum of Mathematics. Sigma}, 10:Paper No. e62, 63, 2022.

\bibitem[Ser72]{Ser72}
Jean-Pierre Serre.
\newblock {Propri\'et\'es galoisiennes des points d'ordre fini des courbes
  elliptiques}.
\newblock {\em Inventiones Mathematicae}, 15(4):259--331, 1972.

\bibitem[Shi71]{Shi71}
Goro Shimura.
\newblock {\em {Introduction to the arithmetic theory of automorphic
  functions}}, volume No. 1 of {\em Kan\^o Memorial Lectures}.
\newblock Iwanami Shoten Publishers, Tokyo; Princeton University Press,
  Princeton, NJ, 1971.
\newblock Publications of the Mathematical Society of Japan, No. 11.

\bibitem[Ste89]{Ste89}
Glenn Stevens.
\newblock {Stickelberger elements and modular parametrizations of elliptic
  curves}.
\newblock {\em Inventiones Mathematicae}, 98(1):75--106, 1989.

\bibitem[Sut16]{Sut16}
Andrew Sutherland.
\newblock {Computing images of Galois representations attached to elliptic
  curves}.
\newblock {\em Forum of Mathematics. Sigma}, 4:Paper No. e4, 79, 2016.

\bibitem[\v{C}18]{C18}
Kestutis \v{C}esnavi\v{c}ius.
\newblock {The Manin constant in the semistable case}.
\newblock {\em Compositio Mathematica}, 154(9):1889--1920, 2018.

\bibitem[WW22]{WW22}
Hanneke Wiersema and Christian Wuthrich.
\newblock {Integrality of twisted $ L $-values of elliptic curves}.
\newblock {\em Documenta Mathematica}, 27:2041--2066, 2022.

\bibitem[Zyw15]{Zyw15}
David Zywina.
\newblock {On the possible images of the mod ell representations associated to
  elliptic curves over Q}, 2015.

\end{thebibliography}

\end{document}